%% file: main.tex
\documentclass[a4paper]{article}

\usepackage[utf8x]{inputenc}
\usepackage[T1]{fontenc}
\usepackage{amssymb}
\usepackage{graphicx}
\usepackage[english]{babel}
\usepackage{amsmath}
\usepackage[amsmath,thmmarks,standard]{ntheorem}
\usepackage{faktor}
\usepackage{xfrac}
\usepackage{stmaryrd}
\usepackage{fancybox}
\graphicspath{{images/}}
\usepackage{hyperref}
\usepackage{geometry}
\usepackage[dvipsnames]{xcolor}
\usepackage{multicol}
\usepackage{stackrel}
\usepackage{parskip}
\usepackage{fancyhdr}
\usepackage[mathscr]{euscript}
\usepackage{thmtools}
\usepackage{mathrsfs}
\usepackage{empheq}
\usepackage{scalerel}
\usepackage{indentfirst}
\usepackage{soul}
\usepackage{caption}
\usepackage{tabularx}
\usepackage{booktabs}
\usepackage[ all ]{xy}

\usepackage{enumitem}
\usepackage{wrapfig}
\DeclareUnicodeCharacter{0308}{HERE!HERE!}
\usepackage{float}
\usepackage{tikz} 
\usetikzlibrary{decorations.pathreplacing}
\usepackage{framed}
\usepackage{comment}
\usepackage{subfiles}
\usepackage{mathtools, nccmath}
\usepackage{listings}
\usepackage{import}
\usepackage{dsfont}
\usepackage[normalem]{ulem}
\usepackage{etoolbox}
\apptocmd{\sloppy}{\hbadness 10000\relax}{}{}

\usepackage{bigints}

\usepackage[affil-it]{authblk}

\usepackage{xparse} \DeclarePairedDelimiterX{\Iintv}[1]{\llbracket}{\rrbracket}{\iintvargs{#1}}
\NewDocumentCommand{\iintvargs}{>{\SplitArgument{1}{,}}m}
{\iintvargsaux#1} %
\NewDocumentCommand{\iintvargsaux}{mm} {#1\mkern1.5mu,\mkern1.5mu#2}

\newtheorem{thm}{Theorem}[section]
\newtheorem{prop}[thm]{Proposition}
\newtheorem{lem}[thm]{Lemma}
\newtheorem{rem}[thm]{Remark}
\newcommand{\cg}{\mathcal{G}}
\newcommand{\cc}{\mathcal{C}}
\newcommand{\ce}{\mathcal{E}}
\newcommand{\cx}{\mathcal{X}}
\newcommand{\cb}{\mathcal{B}}
\newcommand{\crr}{\mathcal{R}}
\newcommand{\ci}{\mathcal{I}}

\newcommand{\hx}{\widehat{x}}

\newcommand{\ind}[1]{\mathbf{1}_{#1}}

\setlist[itemize]{itemsep=10pt, label={$\bullet$}}

\def\e{{\rm e}}
\def\P{\mathbb{P}}
\def\R{\mathbb{R}}
\def\Z{\mathbb{Z}}
\def\E{\mathbb{E}}

{}\def\v{\mathtt{v}}

\def\bkv{B_{k,\mathtt{v}}}
\def\bkmv{B_{k+1,\mathtt{v}}}
\def\bkvm{B_{k,2\mathtt{v}+e}}

\def\akv{A_{k,\mathtt{v}}}
\def\akmv{A_{k+1,\mathtt{v}}}
\def\akvm{A_{k,2\mathtt{v}+e}}

\def\bkpi{B_{k_n, \sigma(i)}}
\def\bkpim{B_{k_n, \sigma(i+1)}}

\def\Xkpi{\widehat{\mathtt{x}}_{k_n, \sigma(i)}}

\def\xkv{\widehat{x}_{k,\mathtt{v}}}
\def\xkmv{\widehat{x}_{k+1,\mathtt{v}}}
\def\xkvm{\widehat{x}_{k,2\mathtt{v}+e}}

\def\Xkv{\widehat{\mathtt{x}}_{k,\mathtt{v}}}
\def\Xkmv{\widehat{\mathtt{x}}_{k+1,\mathtt{v}}}
\def\Xkvm{\widehat{\mathtt{x}}_{k,2\mathtt{v}+e}}

\def\hauteurkm{\e^{(k+1)L\frac{\alpha}{\gamma}}}
\def\hauteurk{\e^{kL\frac{\alpha}{\gamma}}}

\def\hdk{e^{-kdL}}

\def\star{\mathrm{Star}(k,\mathtt{v})}

\def\Cu{C_1}

\def\Cd{C_2}

\def\Ct{C_3}

\def\Dkp{\widehat{D}}
\def\Ckp{\widehat{C}}

\def\Dall{D}
\def\Call{C}

\def\Ctildeu{C^{(1)}}
\def\Ctilded{C^{(2)}}
\def\Ctilde{\tilde{C}}

\def\x{\mathtt{x}}

\def\tboite{\log n}

\def\npboite{(\log n)^{\A}}
\def\npboitecarre{(\log n)^{2\A}}

\def\mnd{\frac{n}{\npboite}}

\def\A{A}

\def\v{\mathtt{v}}
\def\x{\mathtt{x}}

\def\y{\mathtt{y}}

\title{The contact process on Scale-Free Percolation}
\author[1]{Andrée Barnier}
\author[1]{Patrick Hoscheit}
\author[2]{Michele Salvi}
\author[1]{Elisabeta Vergu}
\affil[1]{\small INRAE, MaIAGE,Université Paris-Saclay, 78350 Jouy-en-Josas, France}
\affil[2]{\small Dipartimento di Matematica, Università di Roma Tor Vergata, Via della ricerca scientifica 1, 00133, Rome, Italy}
\date{}
\begin{document}

\maketitle

\begin{abstract}
    We consider the contact process on scale-free percolation, a spatial random graph model where the degree distribution of the vertices follows a power law with exponent $\beta$. We study the extinction time $\tau_{\mathcal G_n}$ of the contact process on the graph restricted to a $d$-dimensional box of volume $n$, starting from full occupancy. In the regime $\beta \in (2, 3)$, where the degrees have finite mean but infinite variance and the graph exhibits the ultra-small world behaviour, we adapt the techniques of \cite{linker_contact_2021} to show that $\tau_{\mathcal G_n}$ is exponential in $n$. Our main contribution, though, deals with the case $\beta > 3$, where the degrees have finite variance and the graph is small-world. We prove that also in this case $\tau_{\mathcal G_n}$ grows exponentially, at least up to a logarithmic correction reflecting the sparser graph structure. The proof requires the generalization of a result from \cite{mountford_exponential_2016} and combines a multi-scale analysis of the graph, the study of the chemical distance between vertices and percolation arguments.
\end{abstract}


%-----------------------------------------------------------------------------------%
%-----------------------------------------------------------------------------------%
        \section{Introduction}\label{sec:intro}
%-----------------------------------------------------------------------------------%
%-----------------------------------------------------------------------------------%

% \input{Sections/1-intro.tex}

\paragraph{The contact process.} 
First introduced in \cite{harris1974contact}, the contact process is a widely studied interacting particle system used to describe the spread of epidemics across a population. Each individual is represented by a node of a network and, at any instant of time, it can be either infected $(1)$ or susceptible $(0)$. More precisely, for $G = (V,E)$ a graph, the contact process $(\xi_{t})_{t \geq 0}$ on $G$ with infection rate $\lambda > 0$ is a continuous-time Markov process on $\{0, 1\}^{V}$ where an infected vertex infects each of its healthy neighbours with rate $\lambda$ and recovers at rate $1$. It is well known that there exists a critical contamination parameter $\lambda_c=\lambda_c(G)$ such that for $\lambda < \lambda_c$ the disease dies out almost surely, and for $\lambda > \lambda_c(G)$ the disease survives with a non-zero probability. On finite graphs, $\lambda_c=\infty$ almost surely, but a similar phase transition can be observed for the \emph{extinction time} $\tau_G =\inf \{ t>0,\ \xi_t=\emptyset \}$, starting from the initial state where all vertices are infected. As $\lambda$ increases, the contact process transitions from a regime of fast extinction ($\tau_G\sim \log|V|$) to a metastability regime, in which $\tau_G$ is at least exponential in the size of the graph. 

The contact process on $\mathbb{Z}^d$ has been extensively studied, with a comprehensive summary available in \cite{liggett1997domination}. In particular, $\lambda_c(\mathbb{Z}^d)>0$, indicating a non-trivial phase transition in the survival of the process. The same critical value also characterises the phase transition for the asymptotic behaviour of the extinction time on finite boxes within $\mathbb{Z}^d$. For $d$-regular trees $\mathbb{T}^d$, it was also shown that $\lambda_c({\mathbb{T}^d})>0$ \cite{pemantle_contact_1992,stacey1996existence}. Interestingly, in this case, there is a phase transition for the extinction time on finite boxes, but it occurs at $\lambda_2(\mathbb{T}^d) > \lambda_c({\mathbb{T}^d})$, where $\lambda_2$ is the critical value for \emph{strong} survival of the contact process \cite{stacey2001contact,cranston2014contact}. By leveraging results on $\mathbb{Z}^d$ and $\mathbb{T}^d$ and employing coarse-graining arguments, the result that $\lambda_c > 0$ was later extended to random regular graphs \cite{mourrat_phase_2016,lalley2017contact} and random graphs with bounded degree \cite{mountford_exponential_2016,chatterjee_contact_2009,mourrat_phase_2016}. 

\paragraph{Scale-free random graphs.} For graphs with unbounded degrees, the behaviour of the process can change: it is possible to have $\lambda_c=0$, so that the epidemic can persist with positive probability regardless of the infection rate. In particular, scale-free random graphs, whose node degree distribution follows a power law, contain highly connected vertices that enable the prolonged existence of the infection. Early studies by  \cite{pastor2001aepidemic,pastor2001bepidemic,pastor2002epidemic} conjectured that for scale-free random graphs with a degree distribution $p(k) \sim k^{-\beta}$, one has $\lambda_c = 0$ when $\beta \leq 3$ and $\lambda_c > 0$ when $\beta > 3$. These conjectures were subsequently invalidated, showing that $\lambda_c=0$ for all $\beta > 0$ on preferential attachment models \cite{berger2005spread} and on the configuration model \cite{chatterjee_contact_2009,mountford2013metastable}. Also, in the seminal work \cite{berger2005spread}, the authors established an explicit relationship between the survival of the epidemic and the presence of highly connected nodes. It was proved that if the epidemic reaches a node with large degree (greater than $\lambda^{-2}$), it survives for a long time. With this, they showed that the extinction time of the contact process on finite restrictions of size $n$ of the graph scales as $\exp(cn^{1/(\beta-1)})$ for some constant $c>0$. Subsequent work by \cite{chatterjee_contact_2009} and \cite{mountford_exponential_2016} analysed metastability on the configuration model and refined the understanding of extinction times for $\beta > 2$, showing that, for any $\lambda > 0$, there exists a constant $c>0$ such that $\mathbb{P} \left(\tau_{G_n} \geq \e^{cn}\right) \to 1$ as $n \to \infty$. 

For graphs with $\lambda_c=0$, it is interesting to look at the rate of decay of the probability of non-extinction when the infection rate $\lambda$ approaches 0. If $v$ is a vertex of the graph, let $\Gamma(\lambda)$ be this probability, starting from the initial state with only the vertex $v$ infected. In \cite{mountford2013metastable}, it was shown that, for a class of Galton-Watson trees corresponding to the local limit of the configuration model with power-law degree distributions with index $\beta>2$,
\begin{equation}\label{eq:non_extinction_mountford}
    \Gamma(\lambda) \asymp 
    \begin{cases}
        \lambda^{\frac{1}{3-\beta}}, & \text{if } \beta \in \left( 2, \frac{5}{2} \right], \\
        \frac{\lambda^{2\beta - 3}}{\log \left( 1/ \lambda \right)^{\beta - 2}}, & \text{if } \beta \in \left(\frac{5}{2}, 3 \right], \\
        \frac{\lambda^{2\beta - 3}}{\log \left( 1/ \lambda \right)^{2\beta - 4}}, & \text{if } \beta \in \left(3, \infty \right).
    \end{cases}
\end{equation}
For $\beta > 3$, \cite{can2015metastability} extended these analyses to preferential attachment models, showing a slightly higher probability of non-extinction due to shorter distances between highly connected nodes, facilitating faster epidemic spread compared to the configuration model. 

\paragraph{Spatial random graphs.} Recently, two papers have considered metastability and survival probabilities for spatial random graphs. For the hyperbolic random graph (HRG), an exponential survival result was proven for all $\lambda>0$ and all curvature parameters $\alpha_\text{HRG}\in (1/2,1)$ \cite[Theorem 1.2]{linker_contact_2021}:
\[
    \P(\tau_{G_n} \ge \e^{cn}) > 1-\exp(-cn^{\delta}),\quad n\ge 1
\]
for some $c>0$ and $\delta\in (0,1)$. The parameter space $\alpha_{HRG}\in (1/2,1)$ corresponds to scale-free networks with $\beta=2\alpha_\text{HRG}+1\in(2,3)$. A similar result was obtained in \cite{gracar_contact_2022} for a class of models characterized by their connection probabilities satisfying:
\[
    (1\wedge (W_x \wedge W_y)^{-\delta\gamma} \|x-y\|^{-\delta d}) \le \P(x\sim y) \le 
    (W_x\wedge W_y)^{-\delta\gamma} (W_x \vee W_y)^{-\delta(\gamma-1)} \|x-y\|^{-\delta d}
\]
for some $\delta>1$ and $\gamma\in(\delta/(\delta+1),1)$, which also yields scale-free networks with index $\beta=1+1/\gamma\in (2,3)$. However, the two models exhibit different behaviour for the non-extinction probability $\Gamma(\lambda)$: while both agree with \eqref{eq:non_extinction_mountford} for $\beta\in (5/2,3)$, the phase transition at $\beta=5/2$ and the strict polynomial regime $\Gamma(\lambda)\sim\lambda^{1/(3-\beta)}$ for $\beta \in (2,5/2)$ only appears in the HRG model. For the models in \cite{gracar_contact_2022}, we actually have the logarithmic correction $\Gamma(\lambda)\sim \lambda^{2\beta - 3}/\log \left( 1/ \lambda \right)^{\beta - 2}$ for the whole ultra-small-world regime $\beta\in (2,3)$. None of these papers deal with the case $\beta\ge 3$, which is the main focus of this paper. 

\paragraph{Scale-free percolation.} This paper will study the contact process on the Scale-Free Percolation (SFP) random graph, which is a model first introduced in \cite{deijfen_scale-free_2013} on the lattice $\mathbb{Z}^d$, and later extended to Poisson point processes on $\R^d$ in \cite{deprez2019scale}, see also \cite{dalmau_scale-free_2019}. Vertices of such a point process get assigned i.i.d.~weights $W_x$ according to a power-law distribution, such as the Pareto distribution. Conditionally on their weights, each couple of vertices $x,y$ gets connected independently by an edge with probability 
\begin{equation}
        p_{x,y} = 1 - \exp \left(-\rho \frac{W_x W_y}{\|x-y\|^\alpha} \right),
\end{equation}
where $\alpha>0$ is a parameter governing long-range spatial interactions between vertices and $\rho>0$ is a percolation parameter regulating the overall density of edges. This model replicates key properties observed empirically in many real-world networks, such as a positive clustering coefficient (a high density of triangles in the graph), scale-freeness (the degree distribution follows a power law of exponent $\beta$) and, under a certain regime of parameters, the small-world phenomenon (distant vertices are connected by short paths on the graph). These properties make it particularly well-suited for modelling certain spatially dependent phenomena in the real world \cite{krioukov2010hyperbolic,barthelemy2022spatial}. SFP belongs to a larger class of models with similar behaviour, called kernel-based random graphs \cite{jorritsma2023cluster,cipriani2025spectrum}, which include the HRG \cite{krioukov2010hyperbolic} and the geometric inhomogeneous random graph (GIRG) \cite{bringmann2019geometric}, and are closely related to the power-law radius geometric random graph \cite{hirsch2017heavy} and the spatial preferential attachment model \cite{aiello2008spatial}. All of these spatial models present a phase transition in the connectivity of the graph at $\beta = 3$, marking the transition between a regime with finite mean and infinite variance for the degree distribution ($\beta \in (2,3)$) and the regime with finite mean and variance ($\beta > 3$). The drastic change in the topology of the graphs as the parameters change impacts the behaviour of stochastic processes defined over these structures, see e.g.~\cite{cipriani2024scale, HHJ17, komjathy2020explosion, komjathy2023four, bansaye2024branching, linker_contact_2021, gracar_contact_2022}.

\subsection{Our contribution}

We will consider the contact process on the SFP random graph with vertex set given by the restriction to a finite box of a Poisson point process in $\R^d$. We will assume the weights on the vertices of the graph to be distributed according to a Pareto distribution of parameter $\tau-1>1$ and we will take $\rho>\rho_c$ so that the graph has a unique infinite component, see Section \ref{subsec:SFP} for a precise description of the model. 

Our main result deals with the extinction time of the contact process:
\begin{thm}\label{main_theorem}
Consider the extinction time $\tau_{\cg_n}$ of the contact process on the SFP random graph restricted to the box $[0,n^{1/d})^d$ in dimension $d\geq 1$. Let $\alpha>d$ and $\tau>1$ and let $\rho>\rho_c$.
    \begin{enumerate}[label=(\roman*),ref=(\roman*)]
        \item\label{(i)} If 
    \(\gamma=\alpha(\tau-1)/d\in (1,2)\), then for all \(\lambda >0\) 
    there exists \(c>0\) such that
   \begin{equation}\label{polya0}
   \lim_{n\to\infty}\P(\tau_{\cg_n} \geq \e^{cn}) = 1\,.
    \end{equation}   
    \item\label{(ii)}    If \(\gamma=\alpha(\tau-1)/d>2\) and \(\alpha\in (d,2d)\), then for all $\lambda > 0$ there exists $c>0$ such that
\begin{equation}\label{polya}
    \lim_{n\to\infty}\P \left( \tau_{\cg_n} \geq \e^{cn(\log n)^{-A}}\right) = 1
\end{equation}
for all \(A>2\gamma/(2-\alpha/d)\).
    \end{enumerate}
\end{thm}
The parameter $\gamma$ (corresponding to $\beta-1$ in the introduction) describes the tail of the degrees of the graph (see \cite{dalmau_scale-free_2019}). Therefore part \ref{(i)} of the theorem covers the regime where the degrees have finite expectation but infinite variance and the graph exhibits the ultra-small-world property (see discussion in Section \ref{subsec:SFP}). Notice that when $\gamma\in (1,2)$ we have $\rho_c=0$ (\cite[Theorem 3.2]{deprez2019scale}), so that \eqref{polya0} holds for all positive values of the percolation parameter $\rho$.

The main contribution of this paper is the proof of part \ref{(ii)}. Here we deal with the range of parameters that guarantee both finite mean and variance of the degree of the nodes, and where the graph presents ``only'' the small-world property. Since the graph is significantly sparser in this regime than in case \ref{(i)}, the approach of \cite{linker_contact_2021} fails, and one has to develop finer techniques to deal with the extinction time. In Section \ref{Techniques}, we survey the main ideas of the proof. 

We also include, for completeness, a result on the non-extinction probability of the process, the proof of which comes again by an extension of the results of \cite{linker_contact_2021}. We sketch this proof in Section \ref{sec:non_extinction}.

\begin{thm}\label{thm:non-extinction} 
Let \(\alpha > d \) and \(\tau>1\) be such that 
    \(\gamma=\alpha(\tau-1)/d\in (1,2)\). As $\lambda \rightarrow 0$, 
    \begin{equation}
        \Gamma(\lambda) \asymp 
        \begin{cases}
            \lambda^{\frac{1}{2-\gamma}}, & \text{ if } \gamma \in \left( 1, 
            \frac{3}{2} \right] \\
            \frac{\lambda^{2\gamma - 1}}{\log \left( 1/ \lambda \right)^{\gamma 
            - 1}} & \text{ if } \gamma \in \left(\frac{3}{2}, 2 \right)
        \end{cases}.
    \end{equation}
\end{thm}

\subsection{Techniques and outline of the paper.}\label{Techniques}

The main idea used in \cite{linker_contact_2021} to show the exponential extinction time of the contact process for any infection rate $\lambda>0$ on a hyperbolic random graph $\cg_n$ of size $n$ is to find a subgraph $G_n$ of $\cg_n$ on which the contact process is easier to analyse. Roughly put, $G_n$ is a tree that contains a set $J_n$ of special nodes, called stars, with a high degree $S$ (cf.~the definition of a \textit{constellation} in Definition \ref{constellation}). The cardinality of $J_n$ must be a positive fraction of $n$, and the stars must be at a distance of at most $D$ from each other on the tree $G_n$. The larger the number of its neighbours, the longer each star can retain the infection (as shown, for example, in \cite{berger2005spread} and \cite{chatterjee_contact_2009}) and can infect other stars that are nearby. So, if $S$ is large enough, then the contact process on $G_n$ (and hence on $\cg_n$) has an extinction time that is exponential in the number of stars. \cite{linker_contact_2021} deals with the regime with $\beta\in(2,3)$ (finite mean and infinite variance of the degree of the nodes), making use of the (very) heavy-tailed degree distribution to construct tree-like constellations with a sufficient number of stars. With a straightforward transformation and some minor adjustments, the HRG model can be transformed into SFP in dimension $d=1$, and so the exponential extinction time can be easily proved also for SFP in the regime $\gamma=\beta-1\in(1,2)$ in dimension $1$. With little work, this can be also generalized to higher dimensions. This is the content of Section \ref{sec:extinction_time_gamma12}, which proves our Theorem \ref{main_theorem}, part \ref{(i)}.

Unfortunately, this approach fails as soon as $\gamma>2$. In this case, the degree distribution has lighter tails, and it is not possible to find a sufficiently high number of stars that are at a fixed distance from each other: the geometric structure of the graph becomes even more relevant in the study of the process. %becomes the main feature of the graphs instead of the structure of high-degree nodes. 
%This is also the case in long-range percolation, which is the limit of SFP when $\tau\to\infty$. 
The first step to prove Theorem \ref{main_theorem}, part \ref{(ii)}, then, is to generalize \cite[Theorem 1.3]{mountford_exponential_2016} allowing the parameters $S=S_n$ and $D=D_n$ to grow with $n$, see Proposition \ref{prop:mountford_strategy}. While the proof of this generalisation only requires minor changes, we rewrite it in a rather self-contained manner with the scope of isolating it and making it independent of the specific choice of the sequence $( \cg_n)_{n\in\mathbb N}$. The rest of the proof is spent to find a sequence of subgraphs $G_n$  of the SFP graphs $\cg_n$ that are constellations where: 
\begin{itemize}
  \item consecutive stars in $J_n$ are at a small graph distance from each other ($D_n\leq c(\log n)^{\nu_p}$);
  \item the stars in $J_n$ have a sufficiently high degree ($S_n\geq c(\log n)^{\nu_s}$ with $\nu_s>\nu_p$) so that a star can retain the infection long enough to be able to pass the infection to the closest stars;
  \item $G_n$ contains a sufficiently high number of such stars ($|J_n|>cn(\log n)^{-A}$).
\end{itemize}
The proof is quite technical and combines a multi-scale analysis of the graph, the study of the chemical distance between vertices and percolation arguments. 

\smallskip

\begin{rem}
    We point out that part \ref{(ii)} of Theorem \ref{main_theorem} yields an exponential lower bound for the extinction time $\tau_{\cg_n}$ with a logarithmic correction $(\log n)^{-A}$ at the exponent, see \eqref{polya}. This comes from the fact that we can only find a set of stars of cardinality of order $n(\log n)^{-A}$, with $A$ carefully chosen, and we don't think it is possible to improve much this lower bound if one follows the present approach. We cannot exclude that a full exponential lower bound $\e^{-cn}$ for $\tau_{\cg_n}$ can be found. If this were not the case, though, the logarithmic correction would be an interesting feature of the small-world regime in contrast with the ultra-small-world regime that has never been observed before, at least up to the authors' knowledge. It is to be noted that a lower bound with a logarithmic correction can be established for any connected graph  when $\lambda>\lambda_c(\Z)$ \cite[Theorem 1.2]{schapira2017extinction}: for all $\epsilon>0$, there exists a universal constant $c_\epsilon>0$ such that, for all connected graphs $G$,
\[
    \P(\tau_G \ge \exp(-c_\epsilon |G|/\log(|G|)^{1+\epsilon}) > c_\epsilon\,.
\]
 By contrast, the result we establish in this work is valid for all $\lambda>0$.
\end{rem} 

The rest of the paper is organised as follows:
\begin{itemize}
  \item In Section \ref{subsec:MainResults} we properly define the SFP random graph model and comment on the different parts of its phase diagram. Afterwards, in Section \ref{subsec:CP_on_stars}, we give the definition of a constellation graph and prove Proposition \ref{prop:mountford_strategy}. 
  \item In Section \ref{sec:extinction_time_gamma2} we prove the main result of the paper, Theorem \ref{main_theorem}, part \ref{(ii)}. In Section \ref{subsec:proof_extinction_gamma2}, we show how to obtain the statement from Proposition \ref{prop:mountford_strategy}. In Section \ref{partition} we lay the foundations for our multiscale analysis. In Section \ref{identification} we explain how to build the constellation subgraph of $\cg_n$ and claim that the probability of finding such a structure in $\cg_n$ tends to $1$ as $n\to\infty$. This is finally proved with the combination of four different lemmas, the proofs of which are carried out in the following four subsections.
  \item Section \ref{sec:extinction_time_gamma12} is dedicated to the proof of Theorem \ref{main_theorem}, part \ref{(i)}.
  \item Finally, in Section \ref{sec:non_extinction} we give a sketch of the proof of Theorem \ref{thm:non-extinction}.
\end{itemize}

\paragraph{Acknowledgements} The authors thank Elisabeta Vergu for her intellectual guidance and contributions to this work. Regrettably, she passed away before the submission of this manuscript. The first author (A.B.) dedicates this paper to her memory, acknowledging her foundational role in shaping this research and starting their academic journey. The authors wish to thank J\'{u}lia Komj\'{a}thy for suggesting this problem and for many interesting discussions. A.B. is supported by the Fondation Mathématique Jacques Hadamard. M.S. is supported by the MUR Excellence Department Project MatMod@TOV, awarded to the Department of Mathematics, University of Rome Tor Vergata, CUP E83C18000100006, and by the MUR 2022 PRIN project GRAFIA, project code 202284Z9E4. M.S. is also part of the INdAM group GNAMPA.

%-----------------------------------------------------------------------------------%
%-----------------------------------------------------------------------------------%
\section{Preliminaries}\label{subsec:MainResults}
%-----------------------------------------------------------------------------------%
%-----------------------------------------------------------------------------------%

%-----------------------------------------------------------------------------------%
\subsection{Properties of Scale-Free Percolation}\label{subsec:SFP}
%-----------------------------------------------------------------------------------%

The Scale-Free Percolation (SFP) random graph $\cg$ in continuous space is constructed as follows. Let $d\ge 1$ and let $\alpha>0,\ \tau>1$ and $\rho>0$. 
\begin{itemize}
    \item First sample a configuration of points $\cx$ according to a homogeneous Poisson point process with a unit intensity on $\mathbb{R}^d$. These constitute the vertex set of the graph.
    \item Then assign to each vertex $x \in \cx$ a random weight $W_x \geq 1$, in such a way that $(W_x)_{x \in \cx}$ is a sequence of i.i.d. random variables following a Pareto distribution on $[1,+\infty)$ with parameter $\tau-1$:
    \begin{equation}\label{eq:weights_definition}
        \P(W_x \geq t) = t^{-(\tau - 1)},\quad t \geq 1.
    \end{equation}
    \item Between every pair of vertices $x, y \in \cx$ an unoriented edge is drawn with probability
    \begin{equation}\label{eq:SFP base}
        p_{x,y} = 1 - \exp \left(-\rho \frac{W_x W_y}{\|x-y\|^\alpha} \right),
    \end{equation}
    where $\|\cdot\|$ is the Euclidean norm and, conditionally on the realisation of the weights, the presence of each edge is independent of the others. 
\end{itemize}

We will sometimes make use of an FKG inequality for inhomogeneous random graphs, which can be stated using an alternative Poissonian construction of SFP. We refer to \cite{gracarRecurrenceTransienceWeightdependent2022} or \cite{heydenreichLaceExpansionMeanField2023} for a full description. Briefly, in this setting, the graph \(\cg\) becomes a deterministic function of a point process \(\xi\) on the space \((\R^d\times[1,\infty))^{[2]}\times[0,1]\), where for a measurable space \(E\) the notation \(E^{[2]}\) represents the space of all sets \(e\subset E\) with cardinality 2, used here to encode potential edges between vertices in $E$. For us, \(\R^d\) represents the spatial coordinates of a vertex, \([1,\infty)\) is the set in which the weights take their values and \([0,1]\) is an independent uniform mark for each edge used to perform the sampling conditionally on positions and weights of its endpoints, that is, an edge is present if the mark associated to it has a value smaller than $p_{x,y}$ of \eqref{eq:SFP base}.
In this formalism, a function \(f\) on the space of point processes on \((\R^d\times[1,\infty))^{[2]}\times[0,1]\) is called \emph{increasing} if it is increasing with respect to set inclusion and value of the weights and decreasing with respect to edge marks. Measurable events \(F\) are increasing if their indicator function \(\mathbf{1}_F\) is increasing. The FKG inequality then states that if \(F,G\) are increasing events, then
\begin{equation}\label{eq:IneqFKG}
    \P(F\cap G) \ge \P(F)\P(G)\,.
\end{equation}
In other words, increasing events are positively correlated.

\begin{figure}[th]
    \centering
    \includegraphics[width=\textwidth]{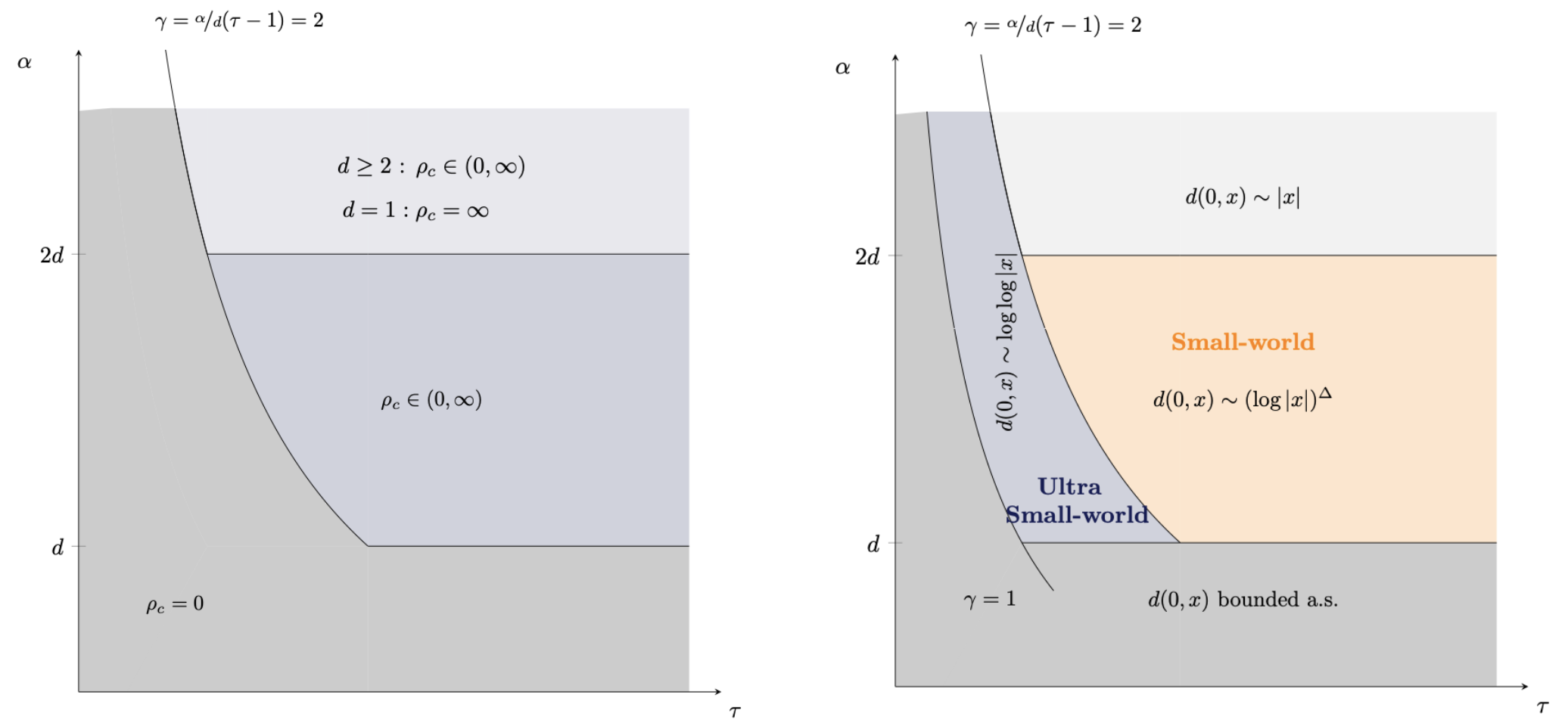}
    \caption{Value of the critical percolation parameter (left) and graph distances (right) for SFP in $\mathbb{R}^d$, $d \geq 1$, for different values of $\alpha$ and $\tau$.}
    \label{fig:regimes}
\end{figure}

Let us take a look now at the behaviour of SFP for different values of its parameters. We denote by $D_x = \{ y \in \cx,\ y \sim x \}$ the degree of a vertex $x$ in the graph. It was shown in \cite[Theorem 2.2]{dalmau_scale-free_2019} that for almost every realization $\cx$ of the Poisson point process and for all $x\in \cx$, there exists a slowly varying function $\ell(\cdot) = \ell (\cdot, \cx, x)$ such that 
\begin{equation*}
    \P( D_x > s \,|\,\cx) = s^{-\gamma}\ell (s).
\end{equation*}
Hence, the degree of each vertex follows a power-law distribution with exponent $\beta = \gamma + 1$. In particular, when $\gamma \in (1,2)$, the degree distribution has a finite mean but infinite variance, while for $\gamma > 2$, both the mean and variance are finite. The threshold $\gamma = 2$ marks a phase transition affecting both the typical graph distance and the percolation behaviour of SFP. Let $\rho_c$ be the critical percolation parameter, defined as
\begin{equation} \label{eq:rho_c}
    \rho_c = \inf \{ \rho > 0,\ \P_0(\vert \mathcal{C}(0) \vert = \infty) > 0 \}\,,
\end{equation}  
where $\P_0$ is the Palm distribution (obtained by adding to the model an additional vertex at the origin) and $\vert \mathcal{C} (0) \vert$ is the cardinality of the connected component containing the origin. For $\gamma \in (1,2)$ and $\alpha > d$, we have $\rho_c = 0$, which implies that an infinite connected component exists with positive probability for any $\rho > 0$. In this regime, the graph distance $d(x, y)$ between two nodes $x$ and $y$ in the infinite component scales as $\log \log(\| x-y \|)$, a phenomenon known as the \emph{ultra-small world} property. When $\gamma > 2$ and $\alpha \in (d, 2d)$, the percolation threshold satisfies $\rho_c > 0$, leading to a phase transition: for $\rho < \rho_c$, there is almost surely no infinite connected component, while for $\rho >\rho_c$ there exists one with positive probability. In this case, edges are more scarce, causing graph distances to scale proportionally to (some power $\Delta$ of) the logarithm of Euclidean distances, a feature referred to as the \emph{small-world} property (see Figure \ref{fig:regimes}). See \cite{deprez2015inhomogeneous, deprez2019scale, hao2023graph, lakis2024improved} for more details on graph distances on SFP.

\textbf{Notation.} From now on, we will call $\cg_n$ the SFP random graph $\cg$ restricted to the box $[0,n^{1/d})^d$, that is, the (almost surely) finite random graph with vertex set $\cx\cap [0,n^{1/d})^d$ and the same set of edges connecting these points in $\cg$. For a given $n$, we let $(\xi_{t})_{t \geq 0}$ be the contact process on $\cg_n$. At any time, we identify a configuration of the contact process $\xi_t$ with the subset $\{x \in V,\ \xi_t(x) = 1\}$ of infected vertices. For a subset of vertices $A$, we note by $\xi_t^{A}$ the contact process with starting configuration where only the vertices in $A$ are infected, i.e.~such that $\xi_0 = A$. For every $n \geq 0$, we let $\tau_{\cg_n} = \inf \{t,\ \xi_t^{\cg_n} = \emptyset \}$ be the extinction time of the contact process $\xi_t^{\cg_n} :=\xi_t^{\cx\cap [0,n^{1/d})^d}$ starting from the whole graph infected.

\begin{figure}[ht]
    \centering
    \includegraphics[width=\textwidth]{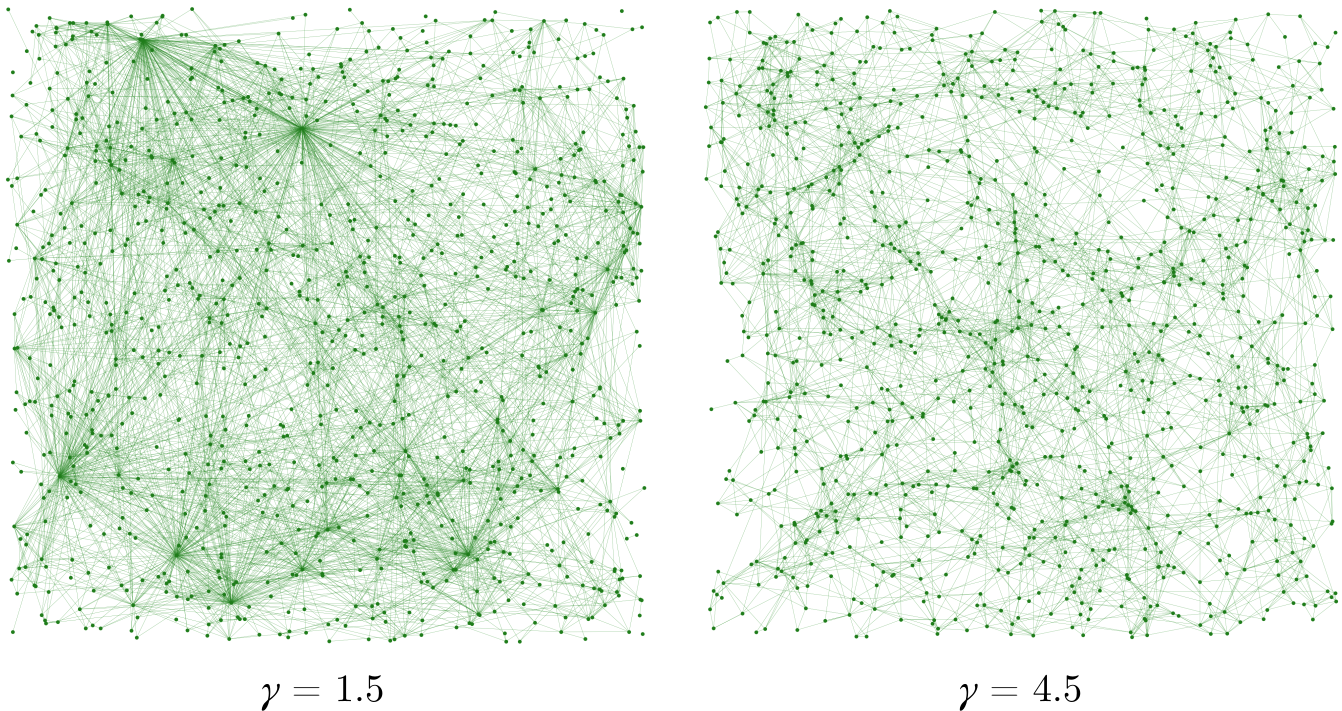}
    \caption{Realizations of SFP with $n=1000$ vertices in a square box $[0,1]^2$ for $\alpha=2.5$, $\tau=2.2$ so that $\gamma=1.5$ (left) and $\alpha=3$, $\tau=4$ so that $\gamma=4.5$ (right). The percolation parameter $\rho$ was chosen such that both graphs had a similar number of edges (average degree approximately equal to 10).}
    \label{fig:ComparaisonGamma}
\end{figure}
    
%-----------------------------------------------------------------------------------%
%-----------------------------------------------------------------------------------%
\subsection{The contact process on constellations}\label{subsec:CP_on_stars}
%-----------------------------------------------------------------------------------%
%-----------------------------------------------------------------------------------%

The present section aims to generalise \cite[Theorem 1.3]{mountford_exponential_2016}. To this end, we first introduce the concept of a constellation graph. If $G=(V,E)$ is a graph, we denote the degree of a node $x \in V$ by $\deg(x)$. Let also $\text{dist}(x, y) $ be the graph distance between any $x,y\in V$. 

\begin{definition}\label{constellation}
Let $S\ge 2, D\ge 1$ and $\Delta\ge 2$. We say that the graph $G$ is a $(S,D,\Delta)$-\emph{constellation} if there exists a set of distinguished vertices $J\subseteq V$ such that the following properties hold:
\begin{enumerate}[label=(P\arabic*)]
    \item $G$ is a connected tree,
    \item For all $x \in J$, $\deg(x) \geq S/2$, 
    \item $\text{dist}(x, y) \leq D$ for all $x, y \in J$ such that $x \overset{*}{\sim} y$, where $x \overset{*}{\sim} y$ indicates that  the only self-avoiding path between $x$ and $y$ does not contain any other vertex of $J$.
    \item\label{degree3} The graph $G' = (V', E')$ given by $V' = J$ and 
\begin{equation*}
    E' = \{\{x, y\}, x, y \in J \text{ and }x \overset{*}{\sim} y \}
\end{equation*}
is a connected tree with degree bounded by $\Delta$. 
\end{enumerate}
\end{definition}

\begin{figure}[ht]
    \centering
    \includegraphics[width=.6\textwidth]{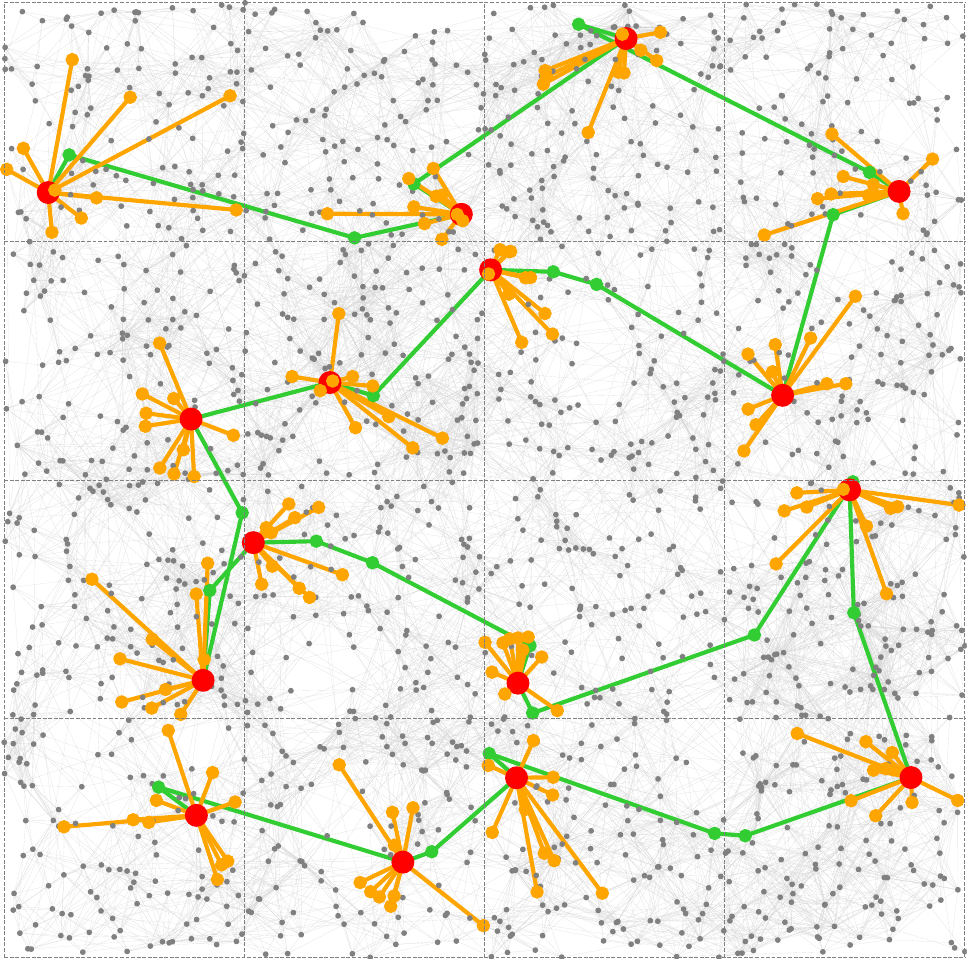}
    \caption{Representation of a (10,3,2)-constellation with $|J|=16$ within a simulation of SFP in dimension 2, obtained with a modified version of the simulator described in \cite{blasius2022efficiently}. The set of distinguished vertices $J$ is drawn in red. For each of these vertices, a set of 10 neighbours is drawn in orange, and the paths linking the distinguished vertices are drawn in green.}
    \label{fig:SFP}
\end{figure}

\begin{prop}\label{prop:mountford_strategy}
    Let $\lambda>0$. Let $(G_n)_{n \in \mathbb{N}}$ be a sequence of 
    $(S_n,D_n,\Delta)$-constellations with a number of vertices growing to infinity as $n\to\infty$, such that $\Delta\ge 2$ and %such that %$S_n$ and 
    %$D_n$ satisfy:
    \begin{equation}\label{star_degree}
        S_n\geq C \lambda^{-2} \log ({1}/{\lambda}) D_n
    \end{equation}
    for some constant $C>0$ large enough. For each $n\in\mathbb N$, let $J_n$ be a corresponding set of distinguished vertices in $G_n$. % satisfing properties 1--4 in Definition \ref{constellation}. 
    % . Suppose that for each $n \in \mathbb{N}$, the graph $G_n$ contains a subgraph $G_n'$ that is a constellation tree with parameters $D_n,S_n$ such that
        % \begin{equation}\label{star_degree}
        % S_n\geq \max\Big\{\lambda^{-2}64 \e^2,\;\frac{7}{c_1} \lambda^{-2} \log ({1}/{\lambda}) D_n\Big\}\,.
        % \end{equation}
        Then there exists a  constant $c > 0$ such that
                \begin{equation*}
                    \lim_{n\to\infty}\P \left(\tau_{{G}_n} \geq \e^{c(\lambda^2 S_n + |J_n|)} \right) = 1\,.%, \hspace{0.6cm} %\text{ as } n \longrightarrow \infty
                \end{equation*}
        where $\tau_{G_n}$ denotes the extinction time of the contact process on $G_n$, starting from full occupancy and with rate of infection $\lambda$.
    \end{prop}

\begin{proof}
Let $G_n'=(V_n',E_n')$ be obtained from $G_n$ and $J_n$ as in Property 4 of Definition \ref{constellation}. Denote by $(\xi_t)_{t \geq 0}$ the contact process on $G_n$ starting from full occupancy. The idea is to couple $(\xi_t)_{t \geq 0}$ with another discrete-time analogue of the contact process $(\eta_r)_{r \in \mathbb{N}}$ on $G_n'$ (see \cite[Section 5]{mountford_exponential_2016}) such that for every $r \in \mathbb{N},$
\begin{equation}\label{eq:comparaison_discrete_continious}
        \P \left( \eta_{r} \neq \emptyset \right) \leq \P \left( \xi_{\kappa r} \neq \emptyset \right)\,,
    \end{equation}
where $\kappa = \exp(c_1\lambda^2 S_n)$, with $c_1$ large enough, is to be thought of as a timescale.

Before constructing $(\eta_r)_{r \in \mathbb{N}}$ we need an auxiliary result, corresponding to \cite[Lemma 6.2]{mountford_exponential_2016}. For $x\in J_n$, let
\begin{align}
    \mathcal N(x)
        &=\{y\in  V_n,\ \text{dist}(x,y)\leq 1\}\\
    \mathcal N'(x)
        &=\{x\}\cup\{y\in  J_n,\ (x, y) \in E_n'\}\\
    \Lambda(x)
        &=\mathcal N(x)\cup\Big(\bigcup_{y\in \mathcal N'(x), y\neq x} \mathcal N(y)\cup b(x,y)\Big)
\end{align}
where, for $x,y\in J_n$ such that $(x, y) \in E_n'$, we denote by $b(x,y)$ the set of vertices of $G_n$ in the unique self-avoiding
path from $x$ to $y$. Furthermore, for a given realization of the contact process and $x\in V_n'$, $r\in\{0,1,\dots\}$, we define the auxiliary process $(\Gamma [x, r]_t)_{r\kappa\leq t\leq (r+1)\kappa}$ defined on $\{0,1\}^{\Lambda(x)}$ as follows. For $r\kappa\leq t\leq (r+1)\kappa$ and $y\in\Lambda(x)$ we put
\begin{align}
    \Gamma [x, r]_t(y)=\ind{\{z\in\mathcal N(x),\ \xi_{r\kappa}(z)=1,\,(z,\kappa r)\longleftrightarrow (y,t) \text{ inside }\Lambda(x)\}}
\end{align}
where ``$(z,\kappa r)\longleftrightarrow (y,t) \text{ inside }\Lambda(x)$'' indicates that there is an infection path from $z$ at time $\kappa r$ to $y$ at time $t$ only involving vertices of $\Lambda(x)$ in the given realization of the contact process. An infection path is meant in the sense of the classical graphical construction for the contact process, see the beginning of \cite[Section 2]{mountford_exponential_2016} for the precise definition.
Given a set of vertices $U$ and $\omega \in \{0,1\}^U$, we say that $U$ is \textit{infested} in $\omega$ if $\vert \{x \in U,\ \omega(x) = 1 \}\vert \geq \tfrac{\lambda}{16 e} | U | $.

\begin{lem}\label{lem:lem6.2.Mountford}
    Consider any $x\in V_n'$. If $\lambda$ is sufficiently small, the following holds. For any $\varepsilon>0$, there exists $n_0=n_0(\varepsilon,\lambda)$ such that, for every $n \geq n_0$,
        \begin{equation}\label{ragno}
            \P \left( \forall y \in \mathcal{N}'(x), \,\mathcal{N}(y) \text{ is infested in } \Gamma[x,r]_{(r+1)\kappa} \;\Big\vert\; \mathcal{N}(x) \text{ infested in } \Gamma[x,r]_{r\kappa} \right) > 1-\varepsilon.
        \end{equation}
\end{lem}

\begin{proof}
    As in the proof of \cite[Lemma 6.2]{mountford_exponential_2016}, we rely on Lemma 3.1 and Lemma 3.2 from \cite{mountford2013metastable}. Lemma 3.1 provides a lower bound for $S_n$ to allow the infection to persist on a star with such degree for at least the duration of the chosen time step. Lemma 3.2 establishes conditions on the distance $D_n$ between stars and their sizes to ensure that the infection can spread effectively from one star to another.

    More specifically, a direct application of \cite[Lemma 3.1 (\textit{ii.})]{mountford2013metastable}, which is possible thanks to \eqref{star_degree}, would yield
    \begin{multline*}
        \P \left(\mathcal{N}(x) \text{ has at least one infected vertex in } \Gamma [x, r]_{(r+1) \kappa}\; \Big|\; \mathcal{N}(x) \text{ infested in } \Gamma [x, r]_{r \kappa}\right) \\ \geq 1 - \exp(-c_1 \lambda^{2} S_n)\,.
    \end{multline*}
    We need more than that. A careful analysis of the proof of \cite[Lemma 3.1 (\textit{ii.})]{mountford2013metastable} shows that there exists a constant $c > 0$ such that
        \begin{equation}\label{eq:lemma3.1}
            \P \left(\mathcal{N}(x) \text{ is infested in } \Gamma [x, r]_{(r+1) \kappa} \;\Big|\; \mathcal{N}(x) \text{ is infested in } \Gamma [x, r]_{r \kappa}\right) \geq 1 - \exp (-c \lambda^{2} S_n)\,.
        \end{equation}
    This comes from the control of the probability of the event $A_{4,j}$ with $j=\kappa$ therein, which is already dealt with in the proof. We skip the details, cf.~also \cite[Equation (6.3)]{mountford_exponential_2016}.

    Now take any $y\in\mathcal N'(x)$. We can use, again thanks to \eqref{star_degree}, \cite[Lemma 3.2]{mountford2013metastable} to show that, for $\lambda$ small enough,
    \begin{equation}\label{eq:lemma3.2}
        \P \left( \exists t \in(r, r+1],\ \mathcal{N}(y) \text{ is infested in } \xi_{t\kappa}  \;\Big|\; \mathcal{N}(x) \text{ infested in } \Gamma [x, r]_{r \kappa}  \right) \geq 1 - \exp (-c_1 \lambda^{2} S_n ).
    \end{equation}
    On the event that $\mathcal{N}(y) \text{ is infested in } \xi_{t\kappa}'$ for some $t \in(r, r+1]$, we further have, by reasoning as for \eqref{eq:lemma3.1}, that $\mathcal{N}(y) \text{ is infested in } \xi_{(r+1)\kappa}$ with probability $1 - \exp \{-\tilde c \lambda^{2} S_n \}$ for some $\tilde c>0$.
    Combining this fact together with (\ref{eq:lemma3.1}) % and (\ref{eq:lemma3.2}), 
    we obtain that there exists a constant $c > 0$ such that
    \begin{equation*}
        \P \left( \mathcal{N}(y) \text{ is infested in } \Gamma[x,r]_{(r+1)\kappa} \;\Big\vert\; \mathcal{N}(x) \text{ infested in } \Gamma[x,r]_{r\kappa} \right) \geq 1 - \exp (-c \lambda^{2} S_n).
    \end{equation*}
    Applying a union bound allows us to conclude the proof of the lemma. 
\end{proof}
We proceed now to the construction of the process $(\eta_r)_{r\in\mathbb N}$.
We introduce the Bernoulli random variables 
    \begin{equation*}
        \{I_{(x,y)}^r ,\ x, y \in V_n' ,\, y \in \mathcal{N}'(x) ,\, r = 0, 1, \dots\}
    \end{equation*}
such that $I_{(x,y)}^r = 1$ if and only if one of the following conditions is satisfied
    \begin{enumerate}
        \item $\mathcal{N}(x)$ is infested in $\xi_{r\kappa }$ and $\mathcal{N}(y)$ is infested in $\Gamma [x, t]_{(r+1)\kappa}$. 
        \item $\mathcal{N}(x)$ is not infested at $\xi_{r\kappa }$. 
    \end{enumerate}
The second condition only ensures that the parameter of the random variables is sufficiently close to $1$ and will be of no particular use. These random variables are not independent, but \cite[Lemma 6.3]{mountford_exponential_2016} guarantees that they stochastically dominate independent Bernoulli random variables of parameter $p>0$:
\begin{lem}
For any $p\in(0,1)$, there exists an $n_0$ such that for all $n\geq n_0$ the following holds. For any $r\in\{0,1,2,\dots\}$ and any $A\in\{0,1\}^{V_n'}$ the random variables $\{I_{(x,y)}^r ,\ x, y \in V_n' ,\, y \in \mathcal{N}'(x)\}$ stochastically dominate $\{\tilde{I}_{(x,y)}^r ,\ x, y \in V_n' ,\, y \in \mathcal{N}'(x)\}$, which are i.i.d.~Bernoulli random variables of parameter $p$.
\end{lem}
The proof of \cite[Lemma 6.3]{mountford_exponential_2016} can be directly adapted to our setting, since it relies exclusively on the fact that $\varepsilon$ in \eqref{ragno} can be taken arbitrarily small (making $n$ grow) and the fact that $G_n'$ has degree uniformly bounded by $\Delta$. 

The process $(\eta_r)_{r=0,1,\dots}$ on $\{0,1\}^{V_n'}$ is now defined as 
$$
\eta_{r+1}(x)=\ind{\{\exists y\in V_n' \text{ with }x\overset{*}{\sim} y,\ \eta_r(y)=1 \text{ and }\tilde{I}_{(y,x)}^r =1\}},\quad  x\in V_n'\,.
$$
Now it is clear by the construction 
% of $(\eta_r)_{r=0,1,2,\dots}$ 
that, since we start with $\xi_0$ equal to $1$ everywhere, if $\eta_r(x)=1$ for some $x$, then $x$ is infested in $\xi_{\kappa r}$. Equation \eqref{eq:comparaison_discrete_continious} then follows.

In order to conclude the proof of Proposition \ref{prop:mountford_strategy}, we make use of \cite[Proposition  5.2]{mountford_exponential_2016} which states that there exists a $c>0$ such that
\begin{equation}\label{poly}
    \lim_{t\to\infty}\P \left( \tau_{\eta} \geq \e^{c|J_n|} \right) =1
\end{equation}
since $(\eta_r)_{r=0,1,2\dots}$ evolves on $T_n'$ which is a tree with 
$|V_n'|=|J_n|$ nodes and degree bounded by $\Delta$, where $\tau_{\eta}$ is the 
extinction time of $(\eta_r)_{r \in \mathbb{N}}$ started from full occupancy. By 
\eqref{eq:comparaison_discrete_continious} and \eqref{poly} and recalling that each 
step of $(\eta_r)_{r \in \mathbb{N}}$ takes a time $\kappa=\exp(c_1\lambda^2S_n)$ for 
$(\xi_t)_{t\geq 0}$ we conclude the proof. \end{proof}

%-----------------------------------------------------------------------------------%
%-----------------------------------------------------------------------------------%
    \section{Proof of Theorem \ref{main_theorem}, part \ref{(ii)}: extinction time for 
    \texorpdfstring{$\gamma > 2$}{gamma >2}}\label{sec:extinction_time_gamma2}
%-----------------------------------------------------------------------------------%
%-----------------------------------------------------------------------------------%

%-----------------------------------------------------------------------------------%
    \subsection{Strategy of the proof}\label{subsec:proof_extinction_gamma2}    
%-----------------------------------------------------------------------------------%

In this section, let %$\alpha\in(d,2d)$, $\rho>\rho_c(d)$ and $\tau\ge 1$ such that 
%$\gamma=\alpha(\tau-1)/d>2$. For all $n\ge 1$, let 
$\cg_n = (V_n, E_n)$ be an SFP random graph  restricted to $[0,n^{1/d})^d$ with parameters $\alpha\in(d,2d)$, $\rho>\rho_c(d)$ and $\tau\ge 1$. Let $A$ be a constant such that \(A>2\gamma/(2-\alpha/d)\). The following proposition is the core of the proof of \ref{(ii)} of Theorem \ref{main_theorem}.

\begin{prop}\label{prop:subgraph_gamma2}
    There exist constants $c_1,c_2,c_3>0$ depending only on $\rho$ such that the following holds. For some $\nu_s>\nu_p>0$ depending only on $A$, $\alpha$ and $\tau$, we can find, with probability going to $1$ as 
$n\to\infty$, a subset $J_n$ of $V_n$ such that
\begin{enumerate}
    \item\label{property1} $|J_n| \geq c_1n(\log n)^{-A}$;
    \item For every $x \in J_n$, $\deg(x) \geq  c_2 
    (\log n)^{\nu_s}$; 
    \item\label{property3} There is an ordering $x_1,x_2,\dots,x_{|J_n|}$ of the points in $J_n$ and there are simple paths $\pi_{i,i+1}$ in $\cg_n$ connecting $x_i$ with $x_{i+1}$  for all $i=1,\dots,|J_n|-1$, such that
    \begin{enumerate}
        \item[3.1] $|\pi_{i,i+1}|\leq c_3 (\log n)^{\nu_p}$\,,
        % $dist_{\cg_n}(x_i,x_{i+1}) \leq c_3 (\log n)^{\nu_p}$\,,
        \item[3.2]\label{paths}  
        All paths are disjoint (except for their endpoints),
        %$\pi_{i,i+1}$ doesn't intersect $\pi_{j,j+1}$ for any $j\neq i-1,i,i+1\,$. 
    \end{enumerate} 
    where for a path $\pi$, we denoted its length by $|\pi|$.
    %\item\label{property4} \red{\sout{For every $x, y \in J_n$, $d_{\cg_n}(x,y) < c_3 
    %\log(n)^{\nu_p}$,}} 
\end{enumerate}
\end{prop}

\begin{proof}[of \ref{(ii)} Theorem \ref{main_theorem}]
    Proposition \ref{prop:subgraph_gamma2} ensures the existence, with probability going to $1$ as $n\to\infty$, of a subgraph $G_n$ of $\cg_n$ that is a $(c_2 (\log n)^{\nu_s},c_3 (\log n)^{\nu_p},2)$-constellation (see Definition \ref{constellation}) with a set of distinguished vertices of size $|J_n|\geq c_1n(\log n)^{-A}$. We point out that the third condition of Proposition \ref{prop:subgraph_gamma2} is needed to  ensure \ref{degree3} of Definition \ref{constellation}.
    
    The remainder of the proof is a straightforward consequence of Proposition  
    \ref{prop:mountford_strategy}. Indeed, for each fixed $\lambda>0$ there exists 
    $n_0$ such that condition \eqref{star_degree} is fulfilled for all $n\geq n_0$.  
    Proposition  \ref{prop:mountford_strategy} then guarantees that there exists 
    $c>0$ such that
    \begin{equation*}
        \lim_{n\to\infty}\P \left(\tau_{{G}_n} \geq \e^{cn(\log n)^{-A}} \right) = 1\,.%, \hspace{0.6cm} %\text{ as } n \longrightarrow \infty
    \end{equation*}
    Since the $G_n$ are subgraphs of $\cg_n$, the same must hold for 
    $\tau_{\cg_n}$.
\end{proof}

The next sections are devoted to the proof of Proposition \ref{prop:subgraph_gamma2}, 
% show that a set $J_n$ with these properties can be found in $\cg_n$ with  probability tending to $1$, 
which will conclude the proof of Theorem \ref{main_theorem} part \ref{(ii)}.

\subsection{A meticulous partition of the box}\label{partition}

To identify the set \(J_n\) of Proposition \ref{prop:subgraph_gamma2}, we begin by defining two subdivisions of the box \([0,n^{1/d})^d\): one coarse and one fine. The \emph{coarse} subdivision is defined as follows. Let 
\[
    m_c = \left\lfloor \frac{n^{1/d}}{(\log n)^{A/d}}\right\rfloor^d.
\]
For every \(\mathtt{v}=(v_1, v_2, \dots, v_d )\in \{0,1,\dots,m_c^{1/d}-1\}^d\) denote by $B_{\mathtt{v}}$ the box in $\R^d$ given by
\begin{equation} \label{eq:SubdivisionGrossiere}
    B_{\mathtt{v}} = \bigtimes_{j=1}^{d} [v_j (\log n)^{A/d} , (v_j + 1) (\log n)^{A/d}).
\end{equation}
To ease the notation, we will label the boxes with indices in $I_c=\{1,2\dots,m_c\}$ in any arbitrary way such that $B_1 = B_{(0,\dots,0)}$ and such that $B_i$ and $B_{i+1}$ are neighbours for all $i\in I_c$. Later on, we will try to find a point of $J_n$ in each of these coarse boxes. Note that the $(B_i)_{i\in I_c}$ form a partition of the slightly smaller cube $[0,m_c^{1/d}(\log n)^{A/d})^d\subset [0,n^{1/d})^d$, but this will be sufficient to construct $J_n$ with the required properties.

We also define a \emph{fine} subdivision of \([0,n^{1/d})^d\), in the spirit of \cite[Section 5]{coppersmith2002diameter}. Let \(\theta\) be such that 
\[
    \max\left(\frac{2}{3},\frac{\alpha}{2d}+\frac{\gamma}{A}\right) < \theta < 1\,,
\]
which is always possible since \(A>2\gamma/(2-\alpha/d)\). Then, for each \(i\in I_c\), divide the cube \(B_i\) into \(m_1 = \lfloor (\log n)^{(1-\theta)A/d}\rfloor^d\) disjoint subcubes of side \((\log n)^{\theta A /d}\) in a way similar  to \eqref{eq:SubdivisionGrossiere}. Again, these subcubes might not cover $B_i$ entirely, but will be sufficient for our purpose. Each one of these cubes will be subdivided further into \(m_2 = \lfloor (\log n )^{(\theta-\theta^2)A/d} \rfloor^d\) subcubes of side \((\log n)^{\theta^2 A/d}\) in the same way. These subdivisions will be repeated 
\[
    s = \left\lfloor \frac{\log (1/2\gamma)}{\log \theta} \right\rfloor
\] 
times, yielding a fine subdivision of \(B_i\) into cubes of volume \((\log n)^{\theta^sA}\). We set \(\nu_p=\theta^s A\) and note that, by definition of \(s\), we have 
\[
    \frac{A}{2\gamma} \le \nu_p  \le \frac{A}{2\gamma\theta}\,.
\]
To formalize the notation, for \(k \in \{1,\dots, s\}\), let \(I_k\) be the set of indices \(\{1, \dots, m_k\}\), with 
\[
    m_k = \left\lfloor (\log n)^{(\theta^{k-1}-\theta^k)A/d} \right\rfloor^d. 
\]
For every \(i \in I_c\) and \(i_1\in I_1,\dots, i_{k-1} \in I_{k-1}\) we define the cubes recursively as follows: \(( B_i^{i_1\dots i_k})_{i_k\in I_k}\) are the disjoint subcubes of volume \((\log n)^{\theta^k A}\), arbitrarily indexed by \(i_k\), that are contained in \(B_i^{i_1\dots i_{k-1}}\) if \(k\ge 2\) or in \(B_i\) if \(k=1\). With this construction, the finest subdivision, containing exclusively disjoint subcubes with volume $(\log n)^{\theta^s A}=(\log n)^{\nu_p}$, has cardinality  
\[
    m_f = m_c\prod_{k=1}^{s} m_k = \left\lfloor \frac{n^{1/d}}{(\log n)^{A/d}} 
    \right\rfloor^d \prod_{k=1}^{s} 
    \lfloor (\log n )^{A (\theta^{k-1}-\theta^k)/d} \rfloor^d.
\]
We have the upper bound
\begin{equation}\label{eq:BorneSupM}
    m_f\le \frac{n}{(\log n)^A}\prod_{k=1}^s (\log n)^{A(\theta^{k-1}-\theta^k)} =  
    \frac{n}{(\log n)^{\nu_p}} \le \frac{n}{(\log n)^{A/2\gamma}}
\end{equation}
as well as the lower bound 
\begin{align}
    m_f & \ge \frac{n}{(\log n)^{\nu_p}}
    \left(1-d\frac{(\log n)^{A/d}}{n^{1/d}}\right) 
    \prod_{k=1}^s\left(1-d(\log n)^{(\theta^k-\theta^{k-1})A/d}\right)\nonumber\\
    & = \frac{n}{(\log n)^{\nu_p}} \left(1-O\left(
    \frac{(\log n)^{A/d}}{n^{1/d}}\right)\right). \label{eq:BorneInfM}
\end{align}
Let then $c_f>0$ be such that $m_f \ge c_f n(\log n)^{-\nu_p}$ for all $n$. For simplicity, we set \(I_f = \{1, \dots, m_f\}\) and denote by $B_i^f$ for $i \in I_f$ all the subcubes in the fine subdivision, numbered arbitrarily. For every $i \in I_c$, we denote by $K_i \subset I_f$ the set of indices such that $k \in K_i$ if and only if $B_k^f \subset B_i$. In the same manner, for every set of indices $i \in I_c,i_1\in I_1,\dots, i_k \in I_k$ with $1\le k \le s$, we denote 
$K_i^{i_1\dots i_k} \subset I_f$ the set of indices such that \(k \in K_i^{i_1\dots i_k}\) if and only if \(B_k^f \subset B_i^{i_1\dots i_k}\). 

\subsection{Identification of the set \texorpdfstring{$J_n$}{Jn}}\label{identification}

For every $i \in I_c$, denote by $\hx_i$ the vertex with maximal weight in $B_i$. We define the following events:
\begin{equation}\label{E_etoile}
    E_\circledast = \{ \forall i \in I_c,\ \hx_i 
    \text{ has at least } c_2 (\log n)^{\nu_s} 
    \text{ neighbours in } B_i \},
\end{equation}
\begin{equation}\label{E_doublefleche}
    E_\leftrightarrow = \begin{Bmatrix*}[l]
        \forall i \in I_c,\ \hx_i \text{ and } 
        \hx_{i+1} \text{ are connected by a path of length}\\
        \text{at most } c_3 
        (\log n)^{\nu_p} \text{ and all these paths are disjoint} \end{Bmatrix*},
\end{equation}  
where $c_2, c_3 > 0$ are constants to be fixed later on, and \(\nu_s\) is a constant such that 
\begin{equation}\label{nu_s}
    \nu_p < \nu_s < \frac{A-1}{\gamma},
\end{equation}
which is possible since \(\nu_p\le A/(2\gamma\theta)<(A-1)/\gamma\) with our choices of \(A\) and \(\theta\). If we show that 
\begin{equation}\label{torrefazione}
\lim_{n\to\infty} \P(E_\circledast\cap E_\leftrightarrow)=1\,
\end{equation} 
then we can choose $J_n=\{\hx_i\}_{i\in I_c}$ and properties \ref{property1}--\ref{property3} of Proposition \ref{prop:subgraph_gamma2} %\ref{subsec:proof_extinction_gamma2} 
will be satisfied with probability going to $1$ as $n\to\infty$. This will conclude the proof of Theorem \ref{main_theorem}, part \ref{(ii)}. The rest of the section is therefore dedicated to the proof of \eqref{torrefazione}.

\smallskip

Given a realisation of $\cg_n$ on $[0,n^{1/d})^d$, we can consider the subgraph given by its restriction to the vertices in $B_i^f$, for each $i \in I_f$. We call $\cc_i^f$ the largest connected component of this subgraph and note
\begin{equation}\label{def:Ci}
    \cc_i = \bigcup_{j \in K_i} \cc_j^f\,.
\end{equation}
For every $i \in I_c,i_1\in I_1,\dots, i_k \in I_k$ we define
\begin{equation}\label{def:Ci_i...}
    \cc_i^{i_1\dots i_k} = \bigcup_{j \in K_i^{i_1\dots i_k}} \cc_j^f.
\end{equation}
Notice that the subgraphs $\cc_i$ and $\cc_i^{i_1\dots i_k}$ do not form a connected component themselves. However, we will later show that they are in fact connected in $\cg_n$ with high probability. 

To prove \eqref{torrefazione}, we define two auxiliary events. Fix some 
\[
    \eta \in \Big(\frac{\alpha\nu_s}{d}, \frac{A-1}{\tau-1}\Big)\,,
\]
which is possible thanks to \eqref{nu_s}. Let
% Notice that on 
% \(E_\circledast\cap E_\leftrightarrow\), the existence of a sub-graph of 
% the form described in Proposition \ref{prop:mountford_strategy} can be guaranteed. 
% Our goal is then to prove that $E_\circledast \cap E_\leftrightarrow$ takes
% place with high probability. 
\begin{align*}
    E_1 & = \left\{  \forall i \in I_f,\ \beta_1 (\log n)^{\nu_p} < |\cc_i^f | < \beta_2 (\log n)^{\nu_p}  \right\}, \\
    E_2 & = \left\{ \forall i \in I_c,\   W_{\hx_i} > (\log n)^\eta\, \text{ and }\, \hx_i \in \cc_{f(i)}^f\right\},
\end{align*}
where $f(i) \in K_i$ indicates the index such that $\hx_i \in B_{f(i)}^f$. 
%The first event controls the size of the largest component in each fine box, whereas the second ensures that the point with the largest weight in each fine box is sufficiently large and that this point belongs to the connected component of its own box. 
The constants $\beta_1, \beta_2 > 0$ will be chosen later. Now, since 
\begin{align*}
    \P (E_{\leftrightarrow} \cap E_{\circledast}) 
    &  \geq \P(E_{\leftrightarrow} \cap E_{\circledast} | E_1 \cap E_2)  \P (E_1 \cap  E_2) \\
    & \ge \left(1-  \P\left(\overline{E}_{\leftrightarrow}  | E_1 \cap E_2 \right)- \P\left( \overline{E}_{\circledast} | E_1 \cap E_2\right)\right) \P(E_2|E_1) \P(E_1),
\end{align*}
equation \eqref{torrefazione} is proved if we show that \(\P(E_{\leftrightarrow}| E_1 \cap E_2)\), \(\P( E_{\circledast} | E_1 \cap E_2))\), \(\P(E_2|E_1)\) and \(\P(E_1)\) all converge to 1 as \(n\to\infty\). The following four lemmas, the proof of which is deferred to the next subsections, conclude therefore the proof of Theorem \ref{main_theorem}, part \ref{(ii)}. 

\begin{lem}\label{Lemma:E_1}
    There exist $\beta_1, \beta_2 > 0$ such that 
    \[
    \lim_{n\to \infty}\P(E_1)=1\,.
    \]
\end{lem}

\begin{lem}\label{Lemma:E_2}
    % For all \(\eta \in (\alpha\nu_s/d, (A-1)/(\tau-1)\)),
    Taking the $\beta_1$ and $\beta_2$ from Lemma \ref{Lemma:E_1}, it holds
    \[
    \lim_{n\to\infty}\P(E_2\,|\,E_1)=1\,.
    \]
    % \[
    %     \P\left(\forall i \in I_c,\ \widehat{x}_i \in 
    % \cc_{f(i)}^f \text{ and } W_{\widehat{x}_i} > 
    % \log(n)^\eta |E_1\right) \to 1
    % \]
\end{lem}

\begin{lem}\label{Lemma:E_3} 
Taking the $\beta_1$ and $\beta_2$ from Lemma \ref{Lemma:E_1}, there exists $c_3>0$ such that
    \[
    \lim_{n\to\infty}\P(E_{\leftrightarrow}\,|\,E_1\cap E_2)=1\,.
    \]
\end{lem}

\begin{lem}\label{Lemma:E_4}
Taking the $\beta_1$ and $\beta_2$ from Lemma \ref{Lemma:E_1}, there exists $c_2>0$ such that
    \[
    \lim_{n\to\infty}\P(E_\circledast\,|\,E_1\cap E_2)=1\,.
    \]
\end{lem}

\subsection{Proof of Lemma \ref{Lemma:E_1}}
Recall that by construction each of the fine subcubes $(B_i^f)_{i\in I_f}$ has volume $(\log n)^{\nu_p}$. Let 
\begin{equation}\label{low}
    E_1^\text{low} = \{  \forall i \in I_f, \,\beta_1 (\log n)^{\nu_p} < |\cc_i^f |\}
\end{equation} 
and 
\begin{equation}\label{upp}
    E_1^\text{upp} = \{  \forall i \in I_f,\, |\cc_i^f| < \beta_2 (\log n)^{\nu_p} \}\,.
\end{equation} 
We will show that both events have probability that goes to $1$ as $n\to\infty$, which implies the same holds for $E_1$.

A union bound for the complementary event of $E_1^\text{low}$ gives
\begin{equation*}
    \P \left(\overline{E_1^\text{low}} \right) \leq m_f \P \left( |\cc_i^f| \leq \beta_1(\log n)^{\nu_p} \right).
\end{equation*}
Since $\rho > \rho_c$, by \cite[Theorem 3.4]{deprez2019scale} there exists $\beta_1 > 0$ such that, for all \(n\) sufficiently large,
\[
    \P\left( |\cc_i^f| \leq \beta_1 (\log n)^{\nu_p}\right) \le
    \exp \left(-\beta_1 (\log n)^{\nu_p(2-{\alpha'}/d)} %((\log n)^{\nu_p/d})^{2d - \alpha'}
    \right)
\]
for any \(\alpha'\in (\alpha,2d)\). Recall that $\nu_p\geq A/2\gamma$ and that $A>2\gamma/(2-\alpha/d)$ by hypothesis, so that \(\nu_p(2-\alpha/d)>1\). Hence we can choose \(\alpha'>\alpha\) such that \(\nu_p(2-\alpha'/d)>1\), too. Using \eqref{eq:BorneSupM} to bound \(m_f\), we obtain
\[
    \P\left(\overline{E_1^\text{low}}\right) \le \frac{n}{(\log n)^{\nu_p}}
    \exp\left(-\beta_1(\log n)^{\nu_p(2-\alpha'/d)}\right),
\]
which converges to 0 since \(\nu_p(2-\alpha'/d)>1\). 

To study the probability of $E_1^\text{upp}$, we simply bound the number of vertices inside each box using the following Chernoff bound for a Poisson random variable $X$ of parameter $\mu$ and $t > \mu$:
\[
    \P \left( X > t \right) \leq \e^{-\mu} \left(\frac{\e \mu}{t}\right)^t.
\]
Since the number of vertices in $B_i^f$ is a Poisson random variable of parameter $(\log n)^{\nu_p}$, by taking $\beta_2>\e$, we get that for any $i \in I_f$ 
\[
    \P (|\cc_i^f| \ge \beta_2 (\log n)^{\nu_p}) 
    \le \e^{-(\log n)^{\nu_p}} 
    \left(\frac{\e}{\beta_2}
    \right)^{\beta_2 (\log n)^{\nu_p}} < \e^{-(\log n)^{\nu_p}}\,.
\]
With this and a union bound, we obtain
\[
    \P \left(E_1^\text{upp}\right) \geq 1 - m_f \exp 
    \left( - (\log n)^{\nu_p} \right),
\]
which converges to 1 using \eqref{eq:BorneSupM} and the fact that \(\nu_p >1\).

\subsection{Proof of Lemma \ref{Lemma:E_2}}

We write $E_2=E_2^1\cap E_2^2$ with
\begin{equation*}
    E_2^1 = \{\forall i \in I_c,\ \ W_{\hx_i} > (\log n)^{\eta}\}\,,\qquad E_2^2 = \{\forall i \in I_c,\ \hx_i \in \cc_{f(i)}^f\}\,.
\end{equation*}
If we prove that $\P (E_2^1 | E_1)$ and $\P (E_2^2 | E_1)$ both converge to $1$ as $n\to\infty$ we are done.
% \begin{equation}\label{gatto}
%     \P (E_2 | E_1) = \P( E_2^1 \cap E_2^2 | E_1) = \P(E_2^2| E_1\cap E_2^1) \P (E_2^1 | E_1)\,.
% \end{equation}

First, we prove that
\begin{equation}\label{gatto1}
    \lim_{n\to\infty}\P (E_2^1 | E_1)=1\,.
\end{equation}
We begin by replacing the conditioning on \(E_1\) with conditioning on \(E_1^\text{low}\) (recall the definition in \eqref{low}). Indeed,
\begin{align}\label{gatto2}
    \P (E_2^1|E_1) & \ge 1-\P\left(\overline{E_2^1} \cup \overline{E_1^\text{low}} \cup \overline{E_1^\text{upp}}\right) \nonumber\\ 
    & \ge 1-\P\left(\overline{E_2^1} \cup \overline{E_1^\text{low}} \right) -\P\left(\overline{E_1^\text{upp}}\right) 
     = \P(E_2^1|E_1^\text{low}) \P(E_1^\text{low}) - \P\left(\overline{E_1^\text{upp}}\right).
\end{align}
We already showed that \(\P(E_1^\text{low})\) and \(\P(E_1^\text{upp})\) converge to 1 as \(n\to \infty\). To show the same for \(\P(E_2^1|E_1^\text{low})\), we use the fact that \(E_1^\text{low}\subset F=\{\forall i \in I_f, |\cx\cap B_i^f| > \beta_1(\log n)^{\nu_p} \}\). Then, reasoning as above, we get 
\begin{align}\label{gatto3}
    \P (E_2^1|E_1^\text{low}) & = \P(E_2^1\cap F|E_1^\text{low}) 
     \ge \P(E_2^1|F)\P(F)-\P\left(\overline{E_1^\text{low}}\right)
\end{align}
Again, \(\P(E_1^\text{low})\) and \(\P(F)\) converge to 1, so we need to show that \(\P(E_2^1|F)\to 1\). For this, first recall that $K_i$ was the set of indices of fine boxes contained in $B_i$ and notice that on the event \(F\) we have, for any \(i\in I_c\),
\begin{equation*}
    |\cx\cap B_i| \geq |K_i| \min_{k \in K_i} |\cx\cap B_i^f| 
    \geq \frac{m_f}{m_c} \beta_1 (\log n)^{\nu_p} \geq  c_f \beta_1 (\log n)^A\,.
\end{equation*}
With this in mind, we can use a union bound and the fact that weights are independently and identically distributed according to a Pareto distribution to obtain
\begin{align*}
    \P (E_2^1 | F) & = 
    1-\P(\exists i\in I_c,\ W_{\hx_i}<(\log n)^\eta\,|\,F)\\
    &\geq
    1 -  m_c \P \big(\forall x \in 
    \mathcal{X}\cap B_i,\ W_x \leq \log(n)^\eta \,|\, F\big)
    \geq 1 - m_c \big( 1- (\log n)^{-\eta (\tau - 1)}\big)^{c_f \beta_1 (\log n)^A}.
\end{align*}
Using \(1-t \leq \e^{-t}\), we conclude that 
\begin{align}\label{croissant}
    \P (E_2^1 | F) 
    %& \geq 1- c_c\frac{n}{(\log n)^A} \exp\left(
    %-(\log n)^{-\eta (\tau - 1)}\right)^{c_f \beta_1 \npboite} \\
     \geq 1- \frac{n}{(\log n)^A} \exp\left(
    -c_f \beta_1 (\log n)^{A-\eta (\tau - 1)}\right),
\end{align}
which converges to $1$ since $\eta < (A-1)/(\tau - 1)$ by definition. Wrapping up, putting together \eqref{gatto2}, \eqref{gatto3} and \eqref{croissant} and the considerations above, we have shown \eqref{gatto1}. 

We are left to show that 
\begin{equation}\label{gattone}
   \lim_{n\to\infty} \P(E_2^2| E_1\cap E_2^1)=1\,,
\end{equation}
which will conclude the proof of the lemma. What we are trying to show is that, with high probability, each $\hx_i$ is in the largest connected component $\mathcal C_{f(i)}^f$ of its corresponding small box. For \(i\in I_c\), write
\begin{align*}
    E_2^2(i)&=\{\hx_i \in \cc^f_{f(i)}\}\\
    E_2^1(i)&=\{W_{\hx_i} > (\log n)^{\eta}\}\\
    E_1(i)&=\{\beta_1(\log n)^{\nu_p}\leq |\mathcal{C}_{f(i)}^f|\leq \beta_2(\log n)^{\nu_p}\}\,.
\end{align*}
Then, for every $i \in I_c$, 
\begin{align}
    \P \left( \overline{E_2^2(i)} \left| E_2^1 \cap E_1\right.\right) &=\P \left( \overline{E_2^2(i)} | E_2^1(i) \cap E_1(i)\right)\nonumber \\
    & \leq \frac{ \P \left( \overline{E_2^2(i)} \cap E_2^1(i) \cap \{| \mathcal{X}\cap  B_{f(i)}^f | \geq \beta_1 (\log n)^{\nu_p}\}\right)}    {\P(E_2^1(i) \cap E_1(i)) }\label{beurre}
\end{align}
since $E_1(i)\subset \{| \mathcal{X} \cap B_{f(i)}^f | \geq \beta_1 (\log n)^{\nu_p}\}$.  The numerator on the r.h.s.~of \eqref{beurre} can be bounded as follows:
\begin{multline}
\P \left( \overline{E_2^2(i)} \cap E_2^1(i) \cap \{|\mathcal{X} \cap B_{f(i)}^f | \geq \beta_1 (\log n)^{\nu_p}\}\right)\\
\leq \sum_{k=\lfloor\beta_1(\log n)^{\nu_p}\rfloor}^{\infty} \P \left( \overline{E_2^2(i)} \, \,\Big|\,\,  E_2^1(i),\, | \mathcal{X} \cap B_{f(i)}^f | = k \right) \P ( | \mathcal{X} \cap B_{f(i)}^f | = k )\,.\label{ramarro}
\end{multline}
We notice now that if $\overline{E_2^2(i)}$ occurs, there must be at least $|\mathcal{X} \cap B_{f(i)}^f|/2$ points to which $\hx_i$ is not connected (otherwise $\hx_i$ would belong to the largest component $\cc^f_{f(i)}$). 

So, on the event $\{|\mathcal{X} \cap B_{f(i)}^f| = k\}$, let 
\[
S_k=\left\{\{y_1, \dots, y_{\lfloor \frac{k}{2}\rfloor}\},\y_j\in \mathcal{X} \cap B_{f(i)}^f\text{ and }y_j\neq y_\ell,\, \forall j\neq \ell,\,\text{ and }y_j\neq\hx_i\right\}
\]
be the set of all $\lfloor k/2\rfloor$-tuples of distinct points in $\mathcal{X} \cap B_{f(i)}^f$. With a union bound, it holds that
\begin{align*}
    \P & \left( \overline{E_2^2(i)} \,\, \Big|\,\,  E_2^1(i),\, | \mathcal{X} \cap B_{f(i)}^f | = k \right) \\
        &\leq \P \left( \exists \{y_1,\dots,y_{\lfloor k/2\rfloor}\}\in S_k\,:  \, \,\hx_i\not\sim y_j \;\;\forall j=1,\dots,\lfloor k/2\rfloor \, \Big|\,  E_2^1(i), | \mathcal{X} \cap  A_{f(i)}^f | = k \right)\\
        &\leq \binom{k}{\lfloor \frac{k}{2} \rfloor} \left( \exp\left( - \rho \frac{(\log n)^{\eta}}{(\sqrt d)^\alpha(\log n)^{\alpha\nu_p/d}}\right)\right)^{\lfloor \frac{k}{2} \rfloor } \\
        &\leq 2^k \exp \left( - c\, (\log n)^{\eta-\alpha\nu_p/d}\left\lfloor \frac{k}{2} \right\rfloor \right)\,,
\end{align*}
with $c=\rho\sqrt d^{-\alpha}$.
Since by definition \(\eta>\alpha\nu_s/d > \alpha\nu_p/d\), there exist 
\(c_1,c_2>0\) such that 
\begin{align*}
    2^k \exp \left( - c (\log n)^{\eta-\alpha\nu_p/d} \left\lfloor \frac{k}{2} \right\rfloor \right)  & \leq \exp \left( - c_1 (\log n)^{\eta-\alpha\nu_p/d } \left\lfloor \frac{k}{2} \right\rfloor\right) \\ 
        & \leq \exp \left( -  c_2(\log n)^{\eta-(\alpha/d - 1)\nu_p} \right)
\end{align*}
uniformly in \(k\ge \beta_1 (\log n)^{\nu_p}\). Putting this back into \eqref{ramarro} yields
\begin{equation}\label{sel}
    \P \left( \overline{E_2^2(i)} \cap E_2^1(i) \cap  \{| \mathcal{X} \cap B_{f(i)}^f | \geq \beta_1 (\log n)^{\nu_p}\}\right) \leq \exp \left( - c_2(\log n)^{\eta-(\alpha/d - 1)\nu_p }\right)\,.
\end{equation}
With a union bound over all $i\in I_c$ and using \eqref{beurre} and \eqref{sel}, we get
\begin{equation*}
    \P(E_2^2|E_1\cap E_{2}^1)=1-\P\left(\bigcup_{i\in I_c} \overline{E_2^2(i)}\Big|E_1\cap E_{2}^1\right) \geq 1-\frac{n}{(\log n)^A}\frac{\exp \left(- c_2 (\log n)^{\eta- (\alpha/d - 1)\nu_p}\right)}{\P(E_2^1 \cap E_1) }\,.
\end{equation*}
Since $\eta-(\alpha/d - 1)\nu_p > \nu_p>1$ and since $\P(E_2^1 \cap E_1)$ converges to $1$, we have \eqref{gattone}.

\subsection{Proof of Lemma \ref{Lemma:E_3}}

Our proof is inspired by that of  \cite[Theorem 2.1]{coppersmith2002diameter}. We need some adjustments to ensure that for each $i\in I_c$, the two paths involving $\hx_i$, linking it to $\hx_{i-1}$ and $\hx_{i+1}$, are disjoint. We will color the fine boxes $(\cc_j^f,\ j\in I_f)$ in red or blue according to a chessboard pattern, and say that a point in $ \bigcup_{i\in I_c}\cc_i$ is blue (red) if it lies in a blue (red) fine box. We will then prove that, with probability going to $1$, there exist paths of length $O((\log n)^{\nu_p})$ linking each couple $\hx_i$, $\hx_{i+1}$ that involve only red vertices contained in the boxes $B_i$ and $B_{i+1}$. By symmetry, the same will be true using only blue vertices. 
We can then build a sequence of disjoint paths as in \eqref{E_doublefleche} by alternating between red and blue paths, which will prove Lemma \ref{Lemma:E_3}. 

%So there are, with probability going to $1$, two sets of disjoint paths with lengths $O(\log(n)^{\nu_p})$ linking the $\hx_i$. We can then build a single set of paths such as in \eqref{E_doublefleche} by alternating between red and blue paths, which will prove Lemma \ref{Lemma:E_3}. 
To be precise, let 
\begin{equation}\label{E_doublefleche_rouge}
    E_\leftrightarrow^{r} = \begin{Bmatrix*}[l]
        \forall i \in I_c,\ \hx_i \text{ and } 
        \hx_{i+1} \text{ are connected by a path of}\\
        \text{red vertices in }B_i\cup B_{i+1}\text{ of length at most } c_3 
        (\log n)^{\nu_p} \end{Bmatrix*},
\end{equation}
and let $E_\leftrightarrow^{b}$ be defined similarly for blue vertices. Since $E_\leftrightarrow^{r}\cap E_\leftrightarrow^{b}\subset E_\leftrightarrow$, it will be sufficient to prove that $\P(E_\leftrightarrow^{r}|E_1\cap E_2)\to 1$ when $n\to\infty$, as the same will be true for $E_\leftrightarrow^{b}$. Thus, for simplicity, in the proof below we will construct only paths involving red vertices.

For all $i \in I_c$ we denote by $K_i^{r} \subset K_i$ the subset of indices such that $j \in K_{i}^{r}$ if and only if $\cc_j^f$ is red. Similarly, for all $1 \leq k \leq s$ and all $i \in I_c$, $\ci=(i_1, \dots, i_{k})\in I_1\times\dots\times I_k$, denote by $K_{i}^{\ci,r} \subset K_{i}^{\ci}$ the indices of the red connected components included in the the box $B_i^{i_1,\dots,i_k}$. For $i \in I_c$, let then
\begin{equation*}
    \crr_i = \bigcup_{j \in K_i^{r}} \cc_j^f
\end{equation*}
be the union of all largest connected components of fine red boxes in $B_i$ and for all $1 \leq k \leq s$ and all $i \in I_c$, $\ci=(i_1, \dots, i_{k})\in I_1\times\dots\times I_k$, let
\begin{equation*}
    \crr_i^{\ci} = \bigcup_{j \in K_i^{\ci,r}} \cc_j^f.
\end{equation*}
Let us define the following events:
\begin{equation*}
    \ce_0 = \{\exists i \in I_c,\
    \crr_i \text{ and } \crr_{i+1} \text{ are not connected by an edge}\}, 
\end{equation*}
and for every $i \in I_c$, 
\begin{align*}
    \ce_i^{1} & = \left\{ \exists j, k \in I_1,\  \crr_i^{j} \text{ and } \crr_i^{k} \text{ are not connected by an edge}\right\}, \\
    \ce_i^{2} & = 
    \begin{Bmatrix*}[l]
        \exists i_1 \in I_1, j, k \in I_2,\ 
        \crr_i^{i_1j} \text{ and } \crr_i^{i_1k} \text{ are not connected by an edge}
    \end{Bmatrix*} \\
    \vdots & \\
    \ce_i^{s}  & = 
    \begin{Bmatrix*}[l]
        \exists i_1 \in I_1, \dots , i_{s-1} \in I_{s-1}, j,k \in I_s ,\ \\ \crr_i^{i_1\dots i_{s-1}j}
        \text{ and } \crr_i^{i_1\dots i_{s-1}k} \text{ are not connected by an 
        edge} \\
    \end{Bmatrix*}. 
\end{align*}
When $\hx_i$ belongs to a red box, under $E_2$ it will be automatically connected to a red vertex in its fine box $\cc_{f(i)}^f$. When $\hx_i$ belongs to a blue box instead, we will prove that, with high probability, it is connected with a single edge to the largest connected component of a  neighbouring red fine box. 
To this end, for all $i \in I_c$, let $\hat{\ci}_i=(\hat{\imath}_1, \dots, \hat{\imath}_{s-1})$ be the indices such that $\hx_i \in B_{i}^{\hat{\imath}_1, \dots, \hat{\imath}_{s-1}}$, the box at level $s-1$ containing $\hx_i$. Define:
% {\color{red}*****************A CHANGER
% 
% \begin{align*}
%     \ce_i = 
%     \begin{Bmatrix*}[l]
%         \{\hx_i \text{ is in a blue box and is not connected to any $\cc_j^f$ with $j \in K_i^{\hat{\ci}_i, r}$} \\ \text{such that $B_j^f$ is a neighbour of $B^f_{f(i)}$\\
%     \end{Bmatrix*}. 
%     %\forall j \in K_i^{\hat{\ci}_i, r},\ \hx_i \text{ is not connected by an edge to } \cc_j^f \}.
% \end{align*}
% }
\begin{align*}
    \ce_i  = 
    \begin{Bmatrix*}[l]
        \hx_i \text{ is in a blue box and is not connected by an edge to any $\cc_j^f$ } \\ \text{with $j \in K_i^{\hat{\ci}_i, r}$ such that $B_j^f$ is a box adjacent to $B^f_{f(i)}$}
    \end{Bmatrix*}. 
\end{align*}
%For all $i \in I_c$, let $\hat{\ci}_i=(\hat{\imath}_1, \dots, \hat{\imath}_{s-1})$ be the indices such that $\hx_i \in B_{i}^{\hat{\imath}_1, \dots, \hat{\imath}_{s-1}}$, the box at level $s-1$ containing $\hx_i$. Define:
% \begin{align*}
%     \ce_i = \{\forall j \in K_i^{\hat{\ci}_i, r},\ \hx_i \text{ is not connected by an edge to } \cc_j^f \}.
% \end{align*}
On the event $E_1 \cap E_2$, we claim that if none of the events 
$\ce_0, (\ce_i, \ce_i^{1}, \dots, \ce_i^{s})_{i \in I_c}$ occur, then for all $i$ there exists a path between $\hx_i$ and $\hx_{i+1}$ of length at most $2^{s+2} \beta_2 (\log n)^{\nu_p}$ that uses only red vertices belonging to $B_i\cup B_{i+1}$.

Indeed, to see this, let $i \in I_c$. The fact that $\ce_i^{s}$ does not occur implies that for any set of indices $i_1 \in I_1, \dots, i_{s-1} \in I_{s-1}$, any pair of red connected components $\crr_i^{i_1\dots i_{s-1}j}$ and $\crr_i^{i_1\dots i_{s-1} k}$ are connected by an edge. Thus, for any pair of vertices $x, y \in \crr_i^{i_1\dots i_{s-1}}$,
\begin{equation*}
    d(x, y) \leq 2 \max \big\{ {\rm diam} ( \crr_i^{i_1\dots i_{s-1} j} ), 
    {\rm diam}(\crr_i^{i_1\dots i_{s-1} k} ) \big\} + 1 
    \leq 2 \beta_2 (\log n)^{\nu_p} + 1
\end{equation*}
where ${\rm diam}(\crr_i^{i_1\dots i_s})$ denotes the diameter of the graph $\crr_i^{i_1\dots i_{s}}$, which is bounded from above by \(\beta_2 (\log n)^{\nu_p}\) on the event \(E_1\). Similarly, since $\ce_i^{s-1}$ does not occur either, then for any pair of vertices $x, y \in \crr_i^{i_1\dots i_{s-2}}$,
\begin{equation*}
    d(x,y) \leq 2 (2\beta_2 (\log n)^{\nu_p} + 1)+ 1 
    = 4 \beta_2 (\log n)^{\nu_p} + 3,
\end{equation*}
and where the shortest path can be chosen using exclusively nodes from $\crr_i^{i_1, \dots, i_{s-2}} $, i.e., using exclusively nodes from red connected components. By following this reasoning, if none of the events $\ce_i^{1}, \dots, \ce_i^{s}$ occur, then for any pair of vertices $x, y \in \crr_i$, 
\begin{equation}\label{distance_i}
    d(x, y) \leq 2^s \beta_2 (\log n)^{\nu_p} + 2^s,
\end{equation}
where the bound for the distance is calculated assuming that the path between $x$ and $y$ touches exclusively nodes from $\crr_i$. Notice now that on $\overline{\ce_i}\cap \overline{\ce_{i+1}}$, there exist red vertices $y_i^1$ (resp. $y_{i+1}^1$) connected by an edge to $\hx_i$ (resp. $\hx_{i+1}$) in one of the $\cc_j^f$ such that $j \in K_i^{\hat{\ci}_i}$ (resp. $j \in K_i^{\hat{\ci}_{i+1}}$). Finally, the non-occurrence of $\ce_0$ implies that, for every $i \in I_c$, $\crr_i$ and $\crr_{i+1}$ are connected by an edge, so there exists $y_i^2 \in \crr_i$ and $y_{i+1}^2\in \crr_{i+1}$ that are connected by an edge, so that
\begin{align*}
    d(\hx_i, \hx_{i+1}) & \leq  d(y_i^{1}, y_i^{2}) 
    + d(y_{i+1}^{2}, y_{i+1}^{1})
    + 3.
\end{align*}
Therefore, using \eqref{distance_i}, we conclude that
\begin{equation*}
    d(\hx_i, \hx_{i+1}) \leq 2^{s+1} \beta_2(\log n)^{\nu_p}
    + 2^{s+1} + 3 \leq 2^{s+2} \beta_2 (\log n)^{\nu_p},
\end{equation*}
where this bound was realized by taking a path between $\hx_i$ and $\hx_{i+1}$ using exclusively nodes from red connected components in $B_i\cup B_{i+1}$. Hence,
\begin{equation*}
    \P (E_{\leftrightarrow}^{r} | E_1 \cap E_2) \geq \P \left( \left. \overline{\ce_0} \cap \bigcap_{i\in I_c} \overline{\ce_i}
     \cap \bigcap_{i \in I_c} \left( \overline{\ce_i^{1}} \cap \dots 
    \cap \overline{\ce_i^{s}} \right) \,\right|\, E_1 \cap E_2 \right). 
\end{equation*}
To prove the convergence to $1$ of the right-hand side of this inequality, we will prove separately that 
\begin{equation}\label{cammello1}
\lim_{n\to\infty}\P \left(\left.\bigcap_{i \in I_c} \overline{\ce_i} \,\right|\, E_1 \cap E_2 \right)= 1,
\end{equation}
that 
\begin{equation}\label{cammello2}
\lim_{n\to\infty}\P \left(\overline{\ce_0}| E_1 \cap E_2 \right)=1,
\end{equation}
and 
\begin{equation}\label{cammello3}
    \lim_{n\to\infty}\P \bigg( \bigcap_{i \in I_c} \overline{\ce_i^{1}} \cap \dots \cap \overline{\ce_i^{s}} \,\bigg|\, E_1 \cap E_2 \bigg)= 1\,.
\end{equation} 
The proof will then be complete.

We start with \eqref{cammello1}. Let $i \in I_c$, let $j\in K_i^{\hat{\ci}_i}$ be such that the boxes $B_{f(i)}^f$ and $B_j^f$ are adjacent. Note that there are at least $d$ such boxes, depending on the position of $B_{f(i)}^f$ within the box of order $s-1$. We can then first compute the probability that $\hx_i$ is not connected to the red connected component $\cc_j^f$:
\begin{equation*}
     \P (  \hx_i \text{ not connected to } \cc_j^f  | E_1 \cap E_2 ) \leq \exp \left( -\rho \frac{(\log n)^\eta}{(2d)^{\alpha/2}
    (\log n)^{\nu_p\alpha/d}} \right)^{\beta_1 (\log n)^{\nu_p}},
\end{equation*}
where we used the lower bounds $|\cc_j^f|>\beta_1 (\log n)^{\nu_p}$ and $W_{\hx_i}>(\log n)^\eta$ on the event $E_1\cap E_2$, as well as the distance bound $\| \hx_i - y \| \leq\sqrt{2d} (\log n)^{\nu_p/d}$ for $y \in C_j^f$. Hence, there exists a constant $c > 0$ such that
\begin{equation*}
     \P ( \hx_i \text{ not connected to } \cc_j^f  | E_1 \cap E_2 ) \leq \exp \left( -c (\log n)^{\eta + \nu_p (1-\alpha/d)} \right).
\end{equation*}
Then, we can use the conditional independence of the events $\{\hx_i \text{ not connected to } \cc_j^f\}$ given $(\hx_i,W_{\hx_i})$ to get:
\begin{align*}
    \P( \ce_i |E_1\cap E_2) & \le \exp \left( -c (\log n)^{\eta + \nu_p (1- \alpha/d)} \right)^d \\
    & = \exp\left(-cd(\log n)^{\eta + \nu_p (1-\alpha/d)} \right).
\end{align*}
By using an union bound over $I_c$, we conclude that
\begin{align*}
    \P \left(\left.\bigcap_{i \in I_c} \overline{\ce_i}\ \right| E_1 \cap E_2 \right) & \geq 1 - \P \left(\left.\bigcup_{i \in I_c} \ce_i \ \right| E_1 \cap E_2 \right) \\
    & \ge  1 - \frac{n}{(\log n)^A} \exp\left(-c (\log n)^{\eta + \nu_p (1-\alpha/d)} \right).
\end{align*}
which converges to 1 since $\eta + \nu_p (1-\alpha/d) > \nu_p > 1$.

Continuing with \eqref{cammello2}, using a union bound, for any $i \in I_c$, we have that 
\begin{align}\label{eq:mathcal{E}0}
    \P(\overline{\ce_0} | E_1 \cap E_2)  &= 1 - \P(\ce_0 | E_1 \cap E_2 ) \nonumber\\
    &\geq 1 - m_c \P(\crr_i \text{ and } \crr_{i+1} \text{ are not connected by an edge } | E_1 \cap E_2 ).
\end{align}
For every $i \in I_c$, $\crr_i$ and $\crr_{i+1}$ are not connected if and only if no pair of points in these graphs is connected. Since the event \(\{\crr_i\text{ and } \crr_{i+1}\text{ are not connected}\}\) is decreasing and \(E_2\) is increasing, the FKG inequality \eqref{eq:IneqFKG} gives
\begin{equation*}
    \P(\crr_i \text{ and } \crr_{i + 1}\text{ are not connected}|E_1\cap E_2)
    \leq \P(\crr_i \text{ and } \crr_{i+1} \text{ are not connected} | E_1).
\end{equation*}
Additionally, on the event $E_1 = \{\forall i \in I_f,\ \beta_1 (\log n)^{\nu_p} < |\cc_i^f | < \beta_2 (\log n)^{\nu_p}\}$,
\begin{equation*}
    |\crr_i| \geq |K_i^{r}| \min_{j \in K_i^{r}} |\mathcal C_j^f| \geq  \frac{1}{2} \frac{m_f}{m_c}\beta_1 (\log n)^{\nu_p} \geq \frac{c_f \beta_1}{2} (\log n)^{A}.
\end{equation*}
So using the independence of edges conditionally on the weights we obtain, for some constant $c>0$,
\begin{align*}
    \P (\crr_i \text{ and } \crr_{i+1} \text{ are not connected by an edge } | E_1 ) 
    & \leq \exp \left(-\rho \frac{1}{\sqrt{d}^\alpha 
    (\log n)^{A\alpha/d}} \right)^{c_f^2 \beta_1^2\npboitecarre/4}  \\ 
    & = \exp \left(- c (\tboite)^{A(2-\alpha/d)}\right), 
\end{align*}
where we use the fact that all weights are greater than $1$ and that for every $x_i\in \crr_i$ and $x_{i+1}\in \crr_{i+1}$ we have $\| x_i - x_{i+1}\|\le\sqrt{d}(2(\log n)^{A/d})$. Plugging this into equation \eqref{eq:mathcal{E}0} and recalling that $m_c = \lfloor (n/(\log n)^A)^{1/d}\rfloor^d$, we conclude that
\begin{equation}\label{eq:epsilon0}
    \P (\overline{\ce_0} | E_1 \cap E_2) \geq 1 - \frac{n}{2(\log n)^A}
    \exp \left(- c (\log n)^{A(2 - \alpha/d)} \right)\,.
\end{equation}
The right-hand side converges to $1$ as $n \to \infty$ since $A$ is chosen such that $A(2 - \alpha/d) > 2\gamma > 1$. 

Let's finally turn our attention to \eqref{cammello3}. Again, using \eqref{eq:IneqFKG}, for any $\ell \in \{1, \dots, s\}$, and any set of indices $i_1 \in I_1, \dots, i_{\ell-1} \in I_{\ell-1}$ and $j,k \in I_\ell$,
\begin{multline*}
    \P(\crr_i^{i_1\dots i_{\ell-1}j} \text{ and } \crr_i^{i_1\dots i_{\ell-1}k}
     \text{ are not connected by an edge } | E_1 \cap E_2) \\
     \leq \P(\crr_i^{i_1\dots i_{\ell-1}j} \text{ and } 
     \crr_i^{i_1\dots i_{\ell-1}k} \text{ are not connected by an edge } | E_1).
\end{multline*}
Similarly, the conditioning on $E_1$ gives rise to the following inequality, for some constant $c_1>0$:
\begin{align*}
    \left|\crr_i^{i_1\dots i_{\ell-1}j}\right| >   \left|K_i^{i_1\dots i_{\ell-1}j, r}\right| 
    \min_{j \in K_i^{i_1\dots i_{\ell-1}j,r}} |\mathcal C_j^f| 
    & \ge m_{\ell+1}\dots m_s \beta_1 (\log n)^{\nu_p}/2 \\
    & \geq c_1  (\log n)^{A\theta^{\ell} - \nu_p} \beta_1 (\log n)^{\nu_p} \\
    & \geq  c_1  \beta_1 (\log n)^{A\theta^{\ell}}.
\end{align*}
We can now use the independence of edges conditionally on weights, and get, for some $c>0$,
\begin{align*}
    \P(\crr_i^{i_1\dots i_{\ell-1}j} \text{ and } & \crr_i^{i_1\dots i_{\ell-1}k} 
    \text{ are not connected by an edge } |\, E_1 ) \\
    & \leq \exp \left(- \rho \frac{1}{\sqrt{d}^\alpha 
    \left((\log n)^{\theta^{\ell-1} A/d}\right)^\alpha} 
    \right)^{c_1^2\beta_1^2(\log n)^{2A\theta^{\ell}}} \\
    & = \exp \left(-c (\log n)^{A\theta^{\ell-1} (2 \theta - 
    \alpha/d)} \right),
\end{align*}
where, again, we bound all weights from below by $1$ and use the fact that two points in $\crr_i^{i_1\dots i_{\ell-1}}$ are distanced by at most $\sqrt{d}(\log n)^{\theta^{\ell-1} A/d}$. With a union bound, we then get that for every $i \in I_c$, and $\ell \in \{1, \dots, s\}$,
\begin{align*}
    \P(\ce_i^{\ell} | E_1 \cap E_2) & \leq \frac{1}{2} m_1 m_2 \dots \binom{m_\ell}{2} \exp
    \left(-c(\log n)^{A\theta^{\ell-1}(2\theta-\alpha/d)} \right) \\
    & \leq  (\log n)^{A(1 - \theta)} \dots (\log n)^{2A(\theta^{\ell-1} - \theta^\ell)} \exp \left(
    -c (\log n)^{A\theta^{\ell-1}(2\theta - \alpha/d)} \right) \\
    & \leq (\log n )^{A (1 + \theta^{\ell-1} - 2\theta^\ell)} \exp \left(-c (\log n)^{A\theta^{\ell-1} (2 \theta - \alpha/d)} \right),
\end{align*}
thus 
\begin{align*}
    \P (\ce_i^{1} \cup \dots \cup \ce_i^{s} | E_1 \cap E_2) & \leq \sum_{\ell=1}^s 
    (\log n )^{A} \exp \left(-c 
    (\log n)^{A\theta^{\ell-1} (2 \theta - \alpha/d)} \right) \\
    & \leq s \left(\log n\right)^{A} \exp \left(-c 
    (\log n)^{\nu_p(2 \theta - \alpha/d)} \right)
\end{align*}
since $\theta^\ell A > \theta^s A = \nu_p$ and $1+\theta^{\ell-1}-2\theta^\ell<1$ for all $\ell \in \{1, \dots, s-1\}$. We then have that 
\begin{align*}
    \P \left(\left. \bigcap_{i \in I_c} \overline{\ce_i^{1}} \cap \dots \cap 
    \overline{\ce_i^{s} } \,\right|\, E_1 \cap E_2  \right) 
    %& \geq 1 - \P \left( \bigcup_{i\in I_c} \ce_i^1 \cup \dots \cup 
    %\ce_i^s \bigg| E_1 \cap E_2 \right) \\
    & \geq 1 -  \mnd \P(\ce_i^{1} \cup \dots \cup \ce_i^{s} | E_1 \cap E_2) \\
    & \geq 1 - sn \exp \left(-c (\log n)^{\nu_p(2 \theta - \alpha/d)} \right).
\end{align*}
Recalling that \(\theta > \gamma/A + \alpha/(2d)\), we notice that the r.h.s.~in the previous expression converges to $1$, since \(\nu_p>A/(2\gamma)>1/(2\theta-\alpha/d)\). Thus, we have verified  \eqref{cammello3} and the proof is finished.

%-----------------------------------------------------------------------------------%
%-----------------------------------------------------------------------------------%
\section{Proof of Theorem \ref{main_theorem}, (i): extinction time  for 
\texorpdfstring{$\gamma \in (1,2)$}{1<gamma<2}}\label{sec:extinction_time_gamma12}
%-----------------------------------------------------------------------------------%
%-----------------------------------------------------------------------------------%

In this section, we prove part \ref{(i)} of Theorem \ref{main_theorem}, closely following the approach of \cite{linker_contact_2021} and \cite{gracar_contact_2022}. Let $\alpha > d$, $\rho> 0$ and $\tau\ge 1$ such that $\gamma=\alpha(\tau-1)/d \in (1,2)$ and, for all $n\ge 1$, let $\cg_n = (V_n, E_n)$ be a random graph distributed as SFP restricted to $[0,n^{1/d})^d$ with parameters $\alpha, \tau, \rho$. We will again make use of Proposition \ref{prop:mountford_strategy} to show survival of the contact process. However, in this regime, the connectivity structure of $\cg_n$ is strongly correlated to its heavy-tailed weight structure, which makes it possible to show that there exist $(S,1,2^d+2)$-constellations containing $\Theta(n)$ stars, instead of $\Theta(n(\log n)^{-A})$ as in Section \ref{sec:extinction_time_gamma2}. It will sometimes be useful to visualize SFP as a graph in $\R^d \times (1, +\infty)$, where each vertex $\mathtt{x} = (x, W_x) \in \R^d \times (1, +\infty)$ consists of a spatial component $x$ and its associated weight $W_x$.

\begin{prop}\label{prop:graphe etoile}
    For all $S > 0$, there exists $\varepsilon > 0$ and $b > 0$ such that for sufficiently large $n$, 
    \begin{equation}
        \P \left( \cg_n \text{ contains a } (S,1,2^d+2)-\text{constellation with } bn \text{ stars} \right) > 1- \e^{-n^{\varepsilon}}.
    \end{equation}
\end{prop}

\begin{proof}
We will begin by partitioning the space $\mathbb{R}^{d} \times (1, +\infty)$ into distinct, non-overlapping boxes, each one potentially containing a star of our subgraph. We will then show that these boxes contain stars and can form a connected subgraph with high probability. The overall strategy is similar to the proof of Theorem 1.2 in \cite{linker_contact_2021}, leveraging the existence of enough high-weight nodes to build a constellation within $\cg_n$.
    
Let $0 < a < \frac{1}{\log 2}$ and $\frac{\gamma}{2}\log 2 < L < \log 2$, which is possible since $\gamma \in (1,2)$. Take $k_n = \lfloor a(\log n)/d \rfloor $ and $m_n = \lfloor n^{\left(1-a\log(2)\right)/d} \rfloor^d$. We will construct $k_n$ layers of non-overlapping boxes, where the layer with the highest weight will contain $m_n$ boxes. For every $k \in \{0, 1, 2,\dots, k_n-1\}$ let
\begin{equation*}
    V_k^n = \{0, 1, 2,\dots, m_n^{1/d} 2^{(k_n-1) - k} -1 \}^d,
\end{equation*}
the index set of the boxes in layer $k$. For $k \in \{0, 1, 2,\dots, k_n-1\}$ and  $\v = (v_1, v_2, v_3,\dots, v_d) \in V_k^n$, let 
\begin{equation*}
    \akv = \bigtimes_{i=1}^{d} \left( 2^{k+1} v_i, 2^{k+1}(v_i + 1) \right),\quad 
    \bkv = \akv \times \left( \hauteurk, \hauteurkm \right).
\end{equation*}
Notice that for every $k \in \{0, 1, 2, \dots, k_n-1\}$ and $\v \in  V_k^n$, $\akv$ has volume $2^{d(k+1)}$ and that indeed $A_{k,\v}\subset [0,n^{1/d})^d$. We will say that $\bkmv$ is the \emph{parent box} of the $2^d$ boxes $\bkvm$ for $e \in \{0, 1\}^d.$ See Fig. \ref{fig:Boxes} for a graphical illustration of this construction.

We now want to prove that each one of these boxes contains, with high probability, a star with at least $S$ neighbours and that the star in a parent box is connected to the stars in its children boxes. Let $\cb_{k,\v}^S = \{|\bkv| \ge S+1\}$ and, on this event, let $\Xkv$ be the vertex with the highest weight in $\bkv$. This vertex will be the centre of the star within the box. Define:
\[
    \star = \cb_{k,\v}^S\cap \{\Xkv\text{ has at least } S \text{ neighbours in } \bkv\}.
\]
and
\[
    \{\Xkmv \leftrightarrow \Xkvm\} = \cb_{k+1,\v}^S \cap \cb_{k,2\v+e}^S \cap \{ \Xkmv \text{ and } \Xkvm \text{ are connected by an edge}\}.
\]
\begin{figure}
    \centering
    \input{representation_boxes}
    \caption{Partitioning scheme of $[0,n^{1/d})^d\times (1,\infty)$ used in the proof of Proposition \ref{prop:graphe etoile}, in dimension 1. As the level $k$ increases, the weight layers $(\e^{kL\alpha/\gamma},\e^{(k+1)L\alpha/\gamma})$ become larger to compensate for the increasing scarcity of high-weight vertices. For $k\ge 1$, boxes $B_{k+1,\v}$ have $2^d$ children boxes denoted by $B_{k,2\v+e}$ for $e\in\{0,1\}^d$.}
    \label{fig:Boxes}
\end{figure}
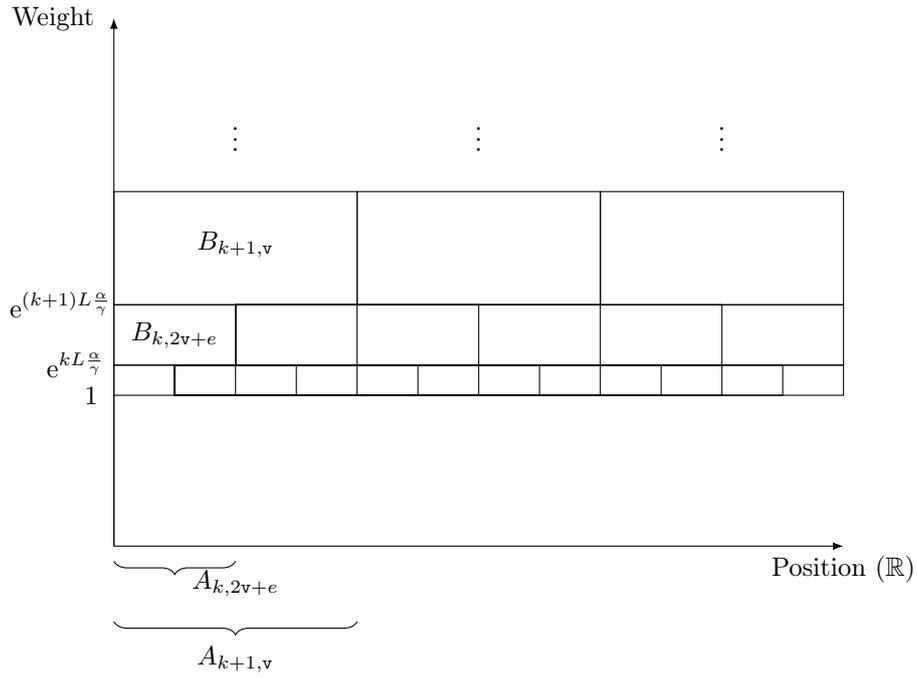

\begin{lem}\label{lemme.todos}
    \begin{enumerate}[label=\roman*)]
        \item There exists $\varepsilon_1 > 0$, $K_1 \ge 1$ and a constant $\Cu > 0$ such that for all $n$ such that $k_n>K_1$:
        \begin{equation}\label{prob:non vide}
            \P \left(\cb_{k,\v}^S \right) > 1 - \e^{-S/2} \exp(-\Cu 
            e^{kd\varepsilon_1 }),\quad K_1< k\le k_n,\ \v \in V_k^n.
        \end{equation}
        \item There exists $K_2 \ge 1$ and a constant $\Cd > 0$ such that for all $n$ such that $k_n>K_2$:
        \begin{equation}\label{prob:voisins}
            \P \left( \star | \cb_{k,\v}^S \right) > 1- \e^{-S/2} \exp(-\Cd 
            \e^{kd\varepsilon_1}),\quad K_2<k\le k_n,\ \v \in V_k^n.
        \end{equation}
        \item There exists $\varepsilon_2 > 0$ and a constant $\Ct> 0$ such that for all $n$ such that $k_n>K_2$:
        \begin{equation}\label{prob:connexion}
            \P \left(\Xkmv \leftrightarrow \Xkvm | \cb_{k+1,\v}^S \cap
            \cb_{k,2\v+e}^S \right) 
            > 1- \exp(- \Ct \e^{kd \varepsilon_2}),\quad K_2<k\le k_n,\ \v \in V_k^n.
        \end{equation}
    \end{enumerate}
\end{lem}

\begin{proof}
The number of vertices in $\bkv$ is Poisson with parameter
\[
    \mu_1^k = \int_{\akv} \int_{\hauteurk}^{\hauteurkm} (\tau -1 )w^{-\tau} dw dz 
    = 2^{(k+1)d} \left( \left(\hauteurk \right)^{-(\tau-1)} - 
    \left( \hauteurkm\right)^{-(\tau-1)} \right).
\]
For $\Cu = \left( 1 - e^{-dL} \right)$ and $\varepsilon_1 = \log(2) - L > 0$, we have:
\begin{equation}\label{eq:mu1}
    \mu_1^k =  2^d \left( 1 - \e^{-dL} \right) 2^{kd} \hdk = \Cu \e^{kd(\log(2) - L)} = \Cu \e^{kd \varepsilon_1}.
\end{equation}

As $k\to\infty$, $\mu_1^k\to\infty$, therefore there exists a $K_1 \ge 1$ such that $\mu_1^k>4S$ for every $k > K_1$. Using a Chernoff bound, we deduce that for every $k > K_1$, 
\begin{equation*}
\P \left( \bkv^S \right) \ge 1-\exp\left( -\frac{(\mu_1^k-S)^2}{2S}\right) \ge 1-\e^{-S/2} \exp(-\mu_1^k),
\end{equation*}
which gives \eqref{prob:non vide}. 

We want to prove now that, conditioned on the fact that $\bkv$ has at least $S+1$ vertices, $\Xkv$ has w.h.p at least $S$ neighbours in $\bkv$. The number of such neighbours follows a Poisson distribution with parameter given by
\begin{equation}\label{mu2inicio}
    \mu_2^k =  \int_{\akv}  \int_{\hauteurk}^{\hauteurkm} (\tau-1) 
    w^{-\tau}  \left(1-\exp \left(-\rho \frac{ W_{\xkv} w  }{\lVert \xkv-y \rVert^\alpha} \right) \right) dw dy.
\end{equation}
Since $\Xkv$ and $\mathtt{y}$ are both in $\bkv$, we have $\lVert \xkv - y\rVert^{-\alpha} \geq d^{-\alpha/2} 2^{-\alpha(k+1)}$. Their weights are bounded below by $\hauteurk$, thus there exists $c>0$ such that
\[
    p_{\Xkv,\mathtt{y}} > 1- \exp\left(- \rho d^{-\alpha/2} 2^{-\alpha(k+1)} \e^{2kL\alpha/\gamma}\right)
    = 1- \exp \left( -\rho c \e^{k\alpha( 2L/\gamma - \log(2))}\right) .
\]
Since $2L > \gamma\log(2)$, this bound can be made independent of $k$ to give
\begin{equation*}
    p_{\Xkv,\mathtt{y}} > 1-\exp \left( - \rho c \e^{\alpha \left( 2L/\gamma - \log(2) \right)}  \right), 
\end{equation*}
for every $k > 1$ and $\mathtt{v} \in V_k$. The integral (\ref{mu2inicio}) can then be bounded as follows, for every $k \in \{1,2, \dots, k_n-1\}$:
\begin{equation*}
    \mu_2^k > \left( 1-\exp \left( - \rho c \e^{\alpha \left( 2L/\gamma - \log(2) \right)} \right) \right) \mu_1^k = C_2 \e^{kd \varepsilon_1}
\end{equation*}
for $C_2 = ( 1-\exp (- \rho c \e^{\alpha ( 2L/\gamma - \log(2))})) C_1$. Using the same Chernoff bound as in the preceding proof, we conclude that there exists a $K_2 \ge 1$ such that for every $k > K_2$, 
\begin{equation*}
    \P (\star | \cb_{k,\v}^S) \geq 1- \e^{-S/2} \exp \left(-C_2 \e^{kd\varepsilon_1}
    \right),
\end{equation*}
which gives \eqref{prob:voisins}.

Turning to the connection probability \eqref{prob:connexion} between boxes, since $\Xkvm$ and $\Xkmv$ are in $\bkmv$, we have that $\lVert \xkv - y \rVert^{-\alpha} \geq d^{-\alpha/2} 2^{-\alpha(k+1)}$. Using the bounds on the weights $W_{\Xkvm}$ and $W_{\Xkmv}$, we get:
\begin{align*}
    \frac{W_{\Xkvm} W_{\Xkmv} }{ \lVert \xkvm-\xkmv \rVert^\alpha} & > \e^{(k+1)L\frac{\alpha}{\gamma}}\e^{kL\frac{\alpha}{\gamma}}d^{-\alpha/2} 
    \e^{-(k+1)\alpha\log(2)} \\    
    & > \e^{L\frac{\alpha}{\gamma} - \alpha\log(2)}d^{-\alpha/2} 
    \e^{k\alpha(2\frac{L}{\gamma} - \log(2)))}.
\end{align*}
We conclude that, for $C_3 = \rho \e^{L\frac{\alpha}{\gamma} - \alpha\log(2)} d^{-\alpha/2}$ and $\varepsilon_2 = \alpha (2L/\gamma - \log(2)) > 0$:
\begin{equation*}
    \P \left(  \Xkmv \xleftrightarrow[]{} \Xkvm | \cb_{k+1,\v}^S \cap \cb_{k,2\v+e}^S \right) 
    > 1-\exp \left(-C_3 \e^{kd \varepsilon_2} \right).
\end{equation*}
\end{proof}

Note that the constants $C_1$, $C_2$ and $C_3$ in Lemma \ref{lemme.todos} depend exclusively on the parameters 
of the random graph model and are independent of $n$, $k$ and $\v$. This lemma allows us to construct, with high probability, a sub-graph consisting of at least two stars, each one with a minimum of $S$ neighbours, where the centres of the stars are given by $\Xkmv$ and $\Xkvm$. Let's now show that, with high probability, this sub-graph contains $\Theta(n)$ stars. We start by proving that the highest-weight vertices of the boxes in the highest layer $k_n$ form a connected subgraph. This allows us to construct a sub-graph where the number of stars is of order $m_n$ (the number of boxes in this layer). We then show that from every $\widehat{\x}_{k_n,\v}$, it is possible to construct a sub-graph with a number of stars of order $2^{k_nd}$. Recalling that $k_n = \lfloor a\log(n)/d \rfloor$ and $m_n = \lfloor n^{\left(1-a\log(2)\right)/d} \rfloor^d$, we have a sub-graph with a number of stars of order
\begin{equation*}
    m_n 2^{k_n d} = \Theta(n^{1-a \log(2)} 2^{a\log(n)})= \Theta(n).
\end{equation*}
The rest of the proof follows similar calculations as presented in \cite{gracar_contact_2022}. To start, let us show that all the boxes in the highest layer form a connected subgraph. To ease the notation, let us forget about the order in space of the boxes in this layer by defining the bijection
\begin{equation*}
    \sigma : \{0, 1, 2, \dots, m_n - 1 \} \longrightarrow V_{k_n}^n
\end{equation*}
such that $\sigma(0) = (0,\dots,0)$ and for all $i$, $\bkpi$ and $\bkpim$ are neighbours. If $\bkpi$ has at least $S+1$ vertices, let $\Xkpi$ be its highest-weight vertex. We will say that the first box $B_{k_n, \sigma(0)}$ is 
\emph{good} if 
\begin{enumerate}
    \item $|B_{k_n, \sigma(0)}|\ge S+1$, 
    \item $\widehat{\mathtt{x}}_{k_n, \sigma(0)}$ has at least $S$ neighbours in 
    $B_{k_n, \sigma(0)}$.
\end{enumerate} 
Moreover, we say that for every $i \in \{1,\dots, m_n-1\}$, the box $\bkpi$ is \emph{good} if 
\begin{enumerate}
    \item $B_{k_n,\sigma(i-1)}$ is good,
    \item $|\bkpi|\ge S+1$,
    \item $\Xkpi$ has at least $S$ neighbours in $\bkpi$,
    \item $\Xkpi$ and $\widehat{\x}_{k_n,\sigma(i-1)}$ are connected by an edge.
\end{enumerate}

We use Lemma \ref{lemme.todos} to prove the following result, stating that the boxes in this layer form a connected subgraph with high probability.

\begin{lem}\label{lemme: Bkpv sont connectés}
    There exists $0< \varepsilon_3 < \varepsilon_1 \wedge \varepsilon_2 \wedge 
    (1/a - \log(2))$ and constants $C,D>0$ such that for all $n$ such that 
    $k_n > K_1\vee K_2$:
    \begin{equation}\label{eq: Bkpv tous connectés}
        \P ( B_{k_n, \v} \text{ is good for every } \v \in V_{k_n}^n) > 1 - m_n
        D\exp \left(-C \e^{k_nd\varepsilon_3}\right).
    \end{equation}
\end{lem}

\begin{proof}
From equations (\ref{prob:non vide}) and (\ref{prob:voisins}) of Lemma \ref{lemme.todos}, there exists constants $C_4$ and $D_4$ such that if $k_n > K_1 \vee K_2$,
\begin{align*}
    \P( B_{k_n, \sigma(0)} \text{ is good}) & = \P (\cb_{k_n, \sigma(0)}^S 
    \cap \{ \widehat{\mathtt{x}}_{k_n, \sigma(0)} \text{ has at least S neighbours 
    in } B_{k_n, \sigma(0)} \} ) \\
    & = \P (\cb_{k_n, \sigma(0)}^S) \P ( \mathrm{Star}(k_n, 0)| 
    \cb_{k_n, \sigma(0)}^S) \\
    & \geq ( 1 - \e^{-S/2} \exp( - C_1 \e^{k_n d\varepsilon_1})) 
    (1- \e^{-S/2} \exp (-C_2 \e^{k_nd\varepsilon_1})) \\
    & \geq  1 - D_4 \exp ( - C_4 \e^{k_nd\varepsilon_1 }).
\end{align*}
Similarly, for $\varepsilon_3 < \varepsilon_1 \wedge \varepsilon_2$ there are constants $\Dkp$ and $\Ckp$ such that if $k_n> K_1\vee K_2$
\begin{align*}
    \P \left( \bkpim \text{ is good} | \bkpi \text{ is good} \right) & > 
    ( 1 - \e^{-S/2} \exp (-C_1 \e^{k_nd\varepsilon_1 }))\\
    & \hspace{1cm}  \times (1- \e^{-S/2} \exp (-C_2 \e^{k_nd\varepsilon_1})) \\
    & \hspace{1cm}  \times (1-\exp(-C_3 \e^{k_nd \varepsilon_2}))\\
    & > 1- \Dkp \exp ( - \Ckp \e^{k_nd\varepsilon_3 }).
\end{align*}
Multiplying these estimates for all boxes in layer $k_n$, and using $(1-x)^k\ge 1-kx$, we get the desired result.
\end{proof}
    
We will also define good boxes in other layers. For $k \in \{0, 1, 2, \dots, k_n-1\}$ and $\v \in V_k^n$ let $\lfloor \frac{\v}{2} \rfloor $ be the vector defined by $(\lfloor \frac{v_1}{2} \rfloor, \dots, \lfloor \frac{v_d}{2} \rfloor)$. With this notation, the parent box of $\bkv$ is $B_{k+1, \lfloor \frac{\v}{2} \rfloor}$. We say that $\bkv$ is \emph{good} if
\begin{enumerate}
    \item $B_{k+1, \lfloor \frac{\v}{2} \rfloor}$ is good,
    \item $|\bkv|\ge S+1$,
    \item $\Xkv$ has at least $S$ neighbours in $\bkv$,
    \item $\Xkv$ and $\widehat{\mathtt{x}}_{k+1, \lfloor \frac{\v}{2} \rfloor}$ are 
    connected by an edge.
\end{enumerate}
With this definition, we can prove similarly that given a good parent box $\bkmv$, the probability that its respective children are good is close to one.
\begin{lem}\label{lemme:Bkmv bonne sachant Bkvm}
    There exist constants $\Dall,\Call>0$ such that for every $k > K_1\wedge K_2$ and $\v \in V_k^n$,
    \begin{equation*}
        \P \left( \bkvm \text{ is good } | \bkmv \text{ is good}\right) 
        > 1 - \Dall \exp \left(- \Call \e^{kd\varepsilon_3}\right)
    \end{equation*}
\end{lem}
At this stage, it remains to show that for each good parent box $\bkmv$, there are enough child boxes, $\bkvm$ with $e \in \{0,1\}^d$, which are also good, thereby ensuring the existence of a total of $\Theta(n)$ good boxes. Denote by $G_k$ the random variable counting the number of good boxes in the layer $k$. Let us now define the following events, for $k > K_1\vee K_2$:
\begin{align*}
    E_k & = \left\{ G_k > 2^d (1 - k^{-2}) G_{k+1} \right\}, \\
    E_{k_n} & = \left\{ \text{Every box in layer } k_n \text{ is good}\right\}.
\end{align*}
For any $k > K_1\vee K_2$, on the event $E_{k} \cap \dots \cap E_{k_n}$, and recalling that there are $m_n$ boxes in layer $k_n$, we have:
\begin{equation*}
    G_k > G_{k_n} \prod_{i=k}^{k_n-1} (1 - i^{-2}) 2^{d(k_n -k)}
    > \left( \prod_{i=2}^{k_n-1} (1 - i^{-2}) \right) m_n  2^{d(k_n -k)}.
\end{equation*}
Now, since $k_n = \lfloor a\log(n)/d \rfloor $ and $m_n = \lfloor n^{\left(1-a\log(2)\right)/d} \rfloor^d$, we have that:
\begin{equation*}
    G_k > \left( \frac{1}{2} \right) n^{1-a \log(2)} 2^{a\log(n)} 2^{-d} 2^{-kd} = 
    c 2^{-kd-1} n.
\end{equation*}
We can conclude that on $E_{k} \cap \dots \cap E_{k_n}$, the total number of good boxes is given by
\begin{equation*}
    G_k + \dots + G_{k_n} > \left( \sum_{j=k}^{k_n} c2^{-jd} \right) n  
    > c 2^{-kd} n = \Theta(n).
\end{equation*}
The realisation of the event $E_k \cap \dots \cap E_{k_n}$ for a $k>K_1\wedge K_2$ thus guarantees the existence of a connected sub-graph in $\cg_n$ containing $\Theta(n)$ stars. To establish Proposition \ref{prop:graphe etoile}, our remaining task is then to prove that $E_k \cap \dots \cap E_{k_n}$ occurs with probability bounded from below by $1-\exp (-n^\varepsilon)$ for some $\varepsilon > 0$. Consequently, the following lemma concludes the proof. 

\begin{lem}\label{lemme: Inter E_k arrive avec grande proba}
There exists $\varepsilon > 0$ and $K_3 \ge K_1\vee K_2$ such that for all $k>K_3$:
\begin{equation*}
    \P \left( E_k \cap \dots \cap E_{k_n}\right) > 1 - \exp(-n^{\varepsilon}).
\end{equation*}
\end{lem}

\begin{proof}
Notice that for every $e$ and $e'$ in $\{0,1\}^d$, with $e \ne e'$, the boxes $\bkvm$ and $B_{k,2\v+e'}$ are disjoint. Hence, the events $\{\bkvm \text{ is good}\}$ and $\{B_{k, 2\v+e'} \text{ is good}\}$ are independent conditionally given $\{\bkmv \text{ is good}\}$. Consequently, we have that conditionally given $G_{k+1}$, $G_k$ is a binomial random variable with parameters $2^dG_{k+1}$ and $p_k$ given by
\begin{equation*}
    p_k = \P \left( \bkvm \text{ is good} | \bkmv \text{ is good}\right) 
    >  1 - \Dall \exp (- \Call \e^{kd\varepsilon_3}).
\end{equation*}
Choose now $K_3 \ge K_1\wedge K_3$ such that $1-k^{-2} < p_k$ for every $k > K_3$. Then, by Chernoff's bound for binomial distributions, the following holds for every $k > K_3$:
\begin{equation*}
    \P\left(\overline{E}_k | G_{k+1} \right) \leq  \exp \left( -\frac{2^{d-1}G_{k+1}
    \Call \e^{kd\varepsilon_3}}{k^2} \right).
\end{equation*}
Consequently, by definition of $k_n$ and $m_n$:
\begin{align*}
    \P \left( E_k| E_{k +1} \cap \dots \cap E_{k_n}\right) & \geq 1 - \exp \left( 
    - \frac{2^{d-1} \left( 2^{d(k_n-1-k)} m_n \prod_{i = k+1}^{k_n-1}(1- i^{-2}) 
    \right) \Call \e^{kd\varepsilon_3} }{k^2}  \right) \\
    & \geq 1- \exp \left( - \frac{1}{2} \prod_{i = 2}^{\infty}(1- i^{-2}) 
    \frac{2^{-kd} 2^{k_nd} m_n \Call \e^{kd\varepsilon_3}}{k^2}  \right) \\
    & \geq 1- \exp \left( - \Ctildeu \frac{2^{-kd} \e^{kd\varepsilon_3}n}{k^2}  
    \right)
\end{align*}
for some $\Ctildeu>0$. Now, since $\varepsilon_3 < \frac{1}{a} - \log2$, for all $k \in \{K_3, \dots, k_n-1\}$:
\begin{equation*}
\frac{2^{-kd} \e^{kd\varepsilon_3} n}{k^2} \geq \frac{2^{-k_nd} n}{k_n^2} 
    \geq \Ctilded \frac{n^{a \varepsilon_3}}{(\log n)^2},
\end{equation*}
hence
\begin{equation*}
    \P \left( E_{k}| E_{k+1} \cap \dots \cap E_{k_n}\right) \geq 1- \exp \left( 
    -\Ctilde \frac{n^{a\varepsilon_3}}{(\log n)^2} \right).
\end{equation*}
Multiplying these estimates, we finally get
\begin{align*}
    \P ( E_k \cap \dots \cap E_{k_n}) & = \P(E_{k_n}) \left( \prod_{i = 1}^{k_n-k}
    \P ( E_{k_n-i} | E_{k_n-i+1} \cap \dots \cap E_{k_n}) \right) \\
        & \geq \P(E_{k_n}) \left(1- \exp \left( - \Ctilde \frac{n^{a\varepsilon_3}}
        {(\log n)^2} \right) \right)^{k_n} \\
        & \geq \P(E_{k_n}) \left( 1- \lfloor a (\log n)/d \rfloor \exp \left( - \Ctilde \frac{n^{a\varepsilon_3}}{(\log n)^2} \right) \right)
\end{align*}
Recall now equation \eqref{eq: Bkpv tous connectés}, which gives us 
\begin{align*}
    \P \left( E_k \cap \dots \cap E_{k_n} \right) 
        & \geq \left( 1 - m_n\exp (-\Call \e^{k_nd\varepsilon_3}) \right) \left( 1- \lfloor a (\log n)/d \rfloor \exp \left( - \Ctilde \frac{n^{ad\varepsilon_3}}{(\log n)^2} \right)\right) \\
    & \geq \left( 1 - n^{1-a\log2} \exp (-\Call n^{ad\varepsilon_3} ) \right) \left( 1- \lfloor a (\log n)/d \rfloor \exp \left( - \Ctilde \frac{n^{ad\varepsilon_3}}{(\log n)^2} \right) \right) \\
    & \geq 1- \exp(-n^{\varepsilon})
\end{align*}
for a certain $\varepsilon > 0$.
\end{proof}

\end{proof}

\begin{proof}[of Theorem \ref{main_theorem}, part \ref{(i)}]
    Let $G_n = (V_n, E_n)$ be the subgraph of $\cg_n$ given by the Proposition \ref{prop:graphe etoile}. $G_n$ is a connected tree by construction. Denote $J_n \subset V_n$ the set of centers of the stars in $G_n$: $J_n= \{\Xkv, k \in \{k_3, \dots, k_n\}, \v \in V_k\}$. Let $G_n'$ be the graph with vertex set $J_n$ and edges given by 
    \begin{equation*}
        E_n' = \left\{\{\mathtt{x}, \mathtt{y} \} : \mathtt{x}, \mathtt{y} \in J_n \text{ and } \mathtt{x} \xleftrightarrow{} \mathtt{y}\right\}.
    \end{equation*}
    $G_n'$ is then a finite tree with degree bounded by $2^d+2$ (the additional $2$ is given to take into account the vertices in the highest layer). Hence, Proposition \ref{prop:mountford_strategy} shows that for every $\lambda > 0$ and for $S \geq C\lambda^{-2} \log(1/\lambda)$, there exists a constant $c> 0 $ such that
        \begin{equation*}
            \lim_{n \to \infty} \P \left( \tau_{G_n} \geq \e^{cn}\right) = 1
        \end{equation*}
    where $\tau_{G_n}$ is the extinction time of the contact process on $G_n$, starting from full occupancy and with rate of infection $\lambda$. Since $G_n$ are subgraphs of $\cg_n$, the same limit must hold for $\tau_{\cg_n}$.
\end{proof}

%------------------------------------------------------------------------------------------------------%
\section{Sketch of the proof of Theorem \ref{thm:non-extinction}: non-extinction probability on the infinite graph}\label{sec:non_extinction}
%------------------------------------------------------------------------------------------------------%

In this section, we outline the proof of Theorem \ref{thm:non-extinction}. We highlight the key arguments required to adapt the proof of \cite[Theorem 1.1]{linker_contact_2021} to the SFP model in $\mathbb{R}^d$. The HRG model, first introduced by \cite{fountoulakis2018law}, is defined over the half-plane $\mathbb{H} = \R \times [0, +\infty)$, and vertices are given by a Poisson point process of intensity 
\begin{equation*}
    d\mu (x, h) = \frac{\tilde{\alpha}}{\pi} \e^{-\tilde{\alpha} h} dx dh,
\end{equation*}
where $\tilde{\alpha}\in(1/2, 1)$. These graphs have power-law degree distribution with exponent $\beta = 2\tilde{\alpha}+1$, so the restriction on $\tilde{\alpha}$ implies that $\beta \in (2, 3)$, which is the regime in which the degree distribution has finite degree and infinite variance. In the HRG setting, $-\tilde{\alpha}^2$ represents the curvature of the hyperbolic plane on which the model is defined. The $x$ component of a vertex can be thought of as its position, and $h$ as its height. In \cite{linker_contact_2021}, $\mathbf{G}_{\infty}$ is defined as the graph where the vertices are given by the above Poisson point process (under the Palm measure with an atom at 0) with i.i.d. heights with density $\tilde{\alpha} \e^{-\tilde{\alpha} h}dh$. An edge between two vertices $v = (x, h)$ and $v' = (x', h')$ is drawn if and only if $\vert x - x' \vert \leq \exp \{(h+h')/2\}.$

The proof of \cite[Theorem 1.1]{linker_contact_2021} is divided into four lemmas, with two lemmas proving the upper and lower bounds in each regime of the parameter $\tilde{\alpha}$. As described in section \ref{sec:intro}, building on the seminal work of the contact process on stars by \cite{berger2005spread} and subsequent results in \cite{mountford2013metastable}, the proof for the lower bounds study the probability that the origin will infect a sufficiently connected node, specifically one with degree larger than $\lambda^{-2}$. Such a node, characterised by its large height, enables the epidemic to persist. Conversely, to establish the upper bounds for each regime, the focus shifts to the event that the epidemic stays within nodes with low degree, i.e. nodes with smaller height. To relate the height of a vertex to its degree, define \cite[Equation (5.1)]{linker_contact_2021} for $h > 0$ and $\tilde{d} > 0$, the functions
\begin{equation*}
    D(h) = \frac{1}{\tilde{\alpha} - \frac{1}{2}} \cdot \e^{h/2}, \hspace{1cm} H(\tilde{d}) = 2 \log \left( \left( \tilde{\alpha} - \frac{1}{2} \right) \cdot \tilde{d} \right)
\end{equation*}
where $D(h)$ corresponds to the expected degree of a vertex at height $h$ and $H(\tilde{d})=D^{-1}(\tilde{d})$ is the height compatible with degree $\tilde{d}$. 

To establish the lower bounds in the first regime, where $\tilde{\alpha} \in (1/2, 3/4]$, the authors show with high probability the existence of a contamination path consisting of nodes with increasing heights, where the first node contaminated by the origin is a vertex whose height exceeds $h_* = H(C \lambda^{- 1/(2 - 2\tilde{\alpha})})$, for $C$ a large constant. This result is achieved by recursively looking at the number of neighbours for each vertex along the path. The tight bound is derived by calculating the probability that the origin successfully contaminates the first ``powerful'' vertex. In the second regime, where $\tilde{\alpha} \in (3/4, 1]$, a much more sparse network structure arises, and contamination occurs indirectly. Here, the origin contaminates a node with height greater than $H( C \log(1/\lambda) \lambda^{-2})$, which then propagates the contamination to a star-like subgraph. A key part of the proof involves establishing the existence of this subgraph, using a subdivision of the space into subspaces, each corresponding to a star.

For the upper bounds, the focus shifts to nodes with low degrees. In the first regime, it is shown that the probability of the contact process remaining within vertices with height below $H(C \lambda^{- 1/(2 - 2\tilde{\alpha})})$, for $C$ a constant, is sufficiently small. This is accomplished by counting all possible contamination paths of this type, using the multivariate Mecke's formula. The proof becomes more intricate in the second regime, as no clear height threshold exists to establish the expected bound for the non-extinction probability directly.

In the context of SFP, the vertex weights naturally correspond to the heights in the hyperbolic graph model, providing a small yet significant step in adapting the proof. However, the main differences between the SFP model and HRG can be summarised in two key points. First, the SFP model is a \emph{soft} model, meaning that edges are formed based on a certain probability, whereas the hyperbolic graph model is a \emph{hard} model, in which the existence of edges is a deterministic function of the distance between nodes. This introduces a higher degree of stochasticity in SFP. The second obvious difference is that SFP can be defined in all dimensions $d \ge 1$, while HRG is the equivalent of a one-dimensional SFP graph, see \cite[Section 9]{komjathy2020explosion} for a detailed description of the bijection in that case. We proceed by showing how these differences are addressed in the proof.

Let the following functions be the analogues of \cite[Equation (5.1)]{linker_contact_2021}. For, $r>1$ and $w > 1$ let
\begin{equation}\label{eq:definition_fonctions_R_W}
    R(w) = w^{d/\alpha} \hspace{1cm}\tilde{W}(r) = r^{\alpha/d}.
\end{equation}  
In this setting, the function $R(w)$ does not correspond to the expected degree of a vertex with weight $w$ exactly, but we will see that a vertex of weight $w$ has an expected degree of order equal to $R(w)$. This remark is the result of \cite[Proposition 3.3]{dalmau_scale-free_2019}, but we give a simpler version to show some integration techniques needed in the rest of the proof. 

For $\x = (x, w_x)$, a vertex in SFP, notice that
\begin{align*}
    \E_{\x}\left[D_{\x}\right] = \E \left[D_{\x} \vert x \in \mathcal{X}, W_x = w_x \right] & = \int_{1}^{+\infty} \int_{\R^d}  (\tau - 1) w^{-\tau} p_{\x, (y, w)} \   dy dw \\
    & = \int_{1}^{+\infty} \int_{\R^d}  (\tau - 1) w^{-\tau} \left( 1 - \exp \left\{-\rho \frac{w_x w}{\lVert x - y \rVert^\alpha} \right\} \right)\   dy dw.
\end{align*}
Contrary to the case of the HRG, to compute the mean degree of a vertex, we need to integrate over all possible positions of the possible neighbours of $\x$ since the model SFP is soft. However, this can be circumvented by noticing that the nodes that are sufficiently close to $\x$, i.e., within a small ball centered in $\x$, are w.h.p. connected to $\x$. Additionally, for nodes outside of this ball, the connection probability with $\x$ can be easily controlled. To see this, for $w \in (1, +\infty)$, let $B_{w}$ be the ball centered at $x$ with radius $(w_x w)^{1/\alpha}$, and let $\overline{B_{w}}$ denote its complement. First, we have that 
\begin{align}\label{eq:lower_bound_degree}
    \E_{\x} \left[D_{\x} \right] & \geq  \int_{1}^{+\infty} \int_{B_w} (\tau - 1) w^{-\tau} \left( 1 - \exp \left(- \rho \right) \right)  dy dw \\ 
    & \geq c w_x^{d/\alpha} \int_{1}^{+\infty} w^{d/\alpha - \tau} \ dw  \geq c w_x^{\frac{d}{\alpha}} = cR(w_x),
\end{align}
for a constant $c > 0$ where the constant changes at each step but never depends on $w$. The last inequality is given since $\gamma = \alpha(\tau -1)/d > 1$, meaning that the exponent of $w$ in the last integral is less than $-1$, and the integral converges. On the other hand, notice that we also have
\begin{equation*}
    \E_{\x} \left[D_{\x} \right] = \int_{1}^{+\infty} (\tau - 1) w_{y}^{-\tau} \left( \int_{B_w}   p_{\x, (y, w)} \   dy  +  \int_{\overline{B_w}} p_{\x, (y, w)} \   dy  \right) dw. 
\end{equation*}

Since $p_{\x, (y, w)} < 1$ and $1 - \exp(-z) \leq z$, the sum of integrals in the last equation can be bounded above by 
\begin{align*}
    \int_{B_w} \, dy +  \int_{\overline{B_w}} \frac{w_\x w}{\lVert x - y \rVert^\alpha} \, dy = (w_\x w)^{\frac{d}{\alpha}} + (w_\x w) C(n) \int_{(w_\x w)^{\frac{1}{\alpha}}}^{+\infty} r^{d-1-\alpha} \, dr,
\end{align*}
where $C(n)$ is the surface measure of the unit sphere $\mathbb{S}^{d-1}$. Since $\alpha > d$, the integral with respect to $r$ exists, and by computing it, we find that 
\begin{align}\label{eq:upper_bound_degree}
    \E_{\x} \left[D_{\x} \right] & \leq  \int_{1}^{+\infty} (\tau - 1) w^{-\tau} \left( (w_\x w)^{d/\alpha} + (w_\x w) C(n) (w_\x w)^{\frac{d - \alpha}{\alpha}}\right) dw \\
    & \leq  C w_x^{d/\alpha} \int_{1}^{+\infty} w^{d/\alpha - \tau} \ dw \leq C w_x^{d/\alpha} = C R(w_x),
\end{align} 
for some positive constant $C$, where we use the last integral's convergence again. With \ref{eq:lower_bound_degree} and \ref{eq:upper_bound_degree}, we conclude that there exists positive constants $c$ and $C$ such that 
    \begin{align}\label{inequality_sup}
        c R(w_x) \leq \E_{\x} \left[D_{\x} \right] \leq C R(w_x).
    \end{align}
With this in hand, the proof of Theorem \ref{thm:non-extinction} follows along the lines of \cite[Theorem 1.1]{linker_contact_2021}. At each step, we use the function $W$ to convert degree thresholds, indicating when the epidemic spreads, into equivalent weight thresholds. To account for the differences between the height distribution in the hyperbolic graph model and the weight distribution in SFP, we map a height $h$ to a weight given by $\exp (\frac{h}{2} \cdot \frac{\alpha}{d})$.

When computing the number of neighbours of a vertex or counting contamination paths, integrals as presented before naturally appear. To bypass the softness of SFP, we use equations \eqref{eq:lower_bound_degree} and \eqref{eq:upper_bound_degree}. Integration techniques from equation \eqref{eq:lower_bound_degree} are applied for the lower bounds, while those from equation \eqref{eq:upper_bound_degree} are used for the upper bounds.

The main distinction between our proof and the one in \cite[Theorem 1.1]{linker_contact_2021} lies in handling \cite[Lemma 4.2]{linker_contact_2021}, which proves the lower bound for the second regime by subdividing the space to construct the star-like subgraph. Specifically, this involves constructing the sets $S_k$, defined as a product of intervals: the first interval corresponds to bounds on the vertex position, and the second to bounds on the weights. However, intervals for the position are unsuitable in the SFP setting. To resolve this, we can use the construction from the proof of \cite[Proposition 2.2]{gracar_contact_2022}, using annuli $A_k$, as illustrated in \cite[Figure 2]{gracar_contact_2022}, instead of intervals.

\bibliographystyle{apalike}
\bibliography{biblio}

\end{document}

%% file: representation_boxes.tex
\begin{tikzpicture}
    % Hauteur de lignes
        \newcommand{\hun}{0.4cm}
        \newcommand{\hdo}{1.2cm}
        \newcommand{\htr}{2.7cm}
        \newcommand{\hze}{3.5cm}

        % \newcommand{\hcu}{0.5}
    % Largeur boites
        \newcommand{\lun}{0.8cm}
        \newcommand{\ldo}{\dimexpr\lun+\lun \relax}
        \newcommand{\ltr}{\dimexpr\ldo+\lun \relax}
        \newcommand{\lcu}{\dimexpr\ltr+\lun \relax}
        \newcommand{\lci}{\dimexpr\lcu+\lun \relax}
        \newcommand{\lse}{\dimexpr\lci+\lun \relax}
        \newcommand{\lsi}{\dimexpr\lse+\lun \relax}
        \newcommand{\loc}{\dimexpr\lsi+\lun \relax}
        \newcommand{\lnu}{\dimexpr\loc+\lun \relax}
        \newcommand{\ldi}{\dimexpr\lnu+\lun \relax}
        \newcommand{\lon}{\dimexpr\ldi+\lun \relax}
        \newcommand{\ldoc}{\dimexpr\lon+\lun \relax}
    % Première ligne
    \draw (0,\hdo) rectangle (\lcu, \htr);
    \draw (\lcu,\hdo) rectangle (\loc, \htr);
    \draw (\loc,\hdo) rectangle (\ldoc, \htr);

    % Deuxième ligne
    \draw (0,\hun) rectangle (\ldo,\hdo);
    \draw (\ldo,\hun) rectangle (\ldo,\hdo);
    \draw (\lcu,\hun) rectangle (\ldo,\hdo);
    \draw (\lse,\hun) rectangle (\ldo,\hdo);
    \draw (\loc,\hun) rectangle (\ldo,\hdo);
    \draw (\ldi,\hun) rectangle (\ldo,\hdo);
    \draw (\ldoc,\hun) rectangle (\ldo,\hdo);

    % Troisième ligne
    % \draw (0,0) node[minimum height=\hun cm,minimum width=\lun cm, draw] {};
    \draw (0,0) rectangle (\lun,\hun);
    \draw (\lun,0) rectangle (\lun, \hun);
    \draw (\ldo,0) rectangle (\lun, \hun);
    \draw (\ltr,0) rectangle (\lun, \hun);
    \draw (\lcu,0) rectangle (\lun, \hun);
    \draw (\lci,0) rectangle (\lun, \hun);
    \draw (\lse,0) rectangle (\lun, \hun);
    \draw (\lsi,0) rectangle (\lun, \hun);
    \draw (\loc,0) rectangle (\lun, \hun);
    \draw (\lnu,0) rectangle (\lun, \hun);
    \draw (\ldi,0) rectangle (\lun, \hun);
    \draw (\lon,0) rectangle (\lun, \hun);
    \draw (\ldoc,0) rectangle (\lun, \hun);

    % Ensemble Ak
    \draw[decorate,decoration={brace,amplitude=5pt}] (\ldo,-2.2) -- (0,-2.2);
    \draw[decorate,decoration={brace,amplitude=5pt}] (\lcu,-3) -- (0,-3);

    %Nodes
    \node at (\lun,0.8) {$\bkvm$};
    \node at (\ldo,2) {$\bkmv$};
    \node at (-0.5,\hun) {$\hauteurk$};
    \node at (-0.7,\hdo) {$\hauteurkm$};
    \node at (-0.3,0) {$1$};
    \node at (\ldo,-2.5) {$\akvm$};
    \node at (\ldo,-3.5) {$\akmv$};

    \node at (\ldo,\hze) {$\vdots$};
    \node at (\lse,\hze) {$\vdots$};
    \node at (\ldi,\hze) {$\vdots$};

    %axes
    \draw (0,0) -- (0,-2);
    \draw[-latex] (0,-2) -- (\ldoc,-2);
    \draw[-latex] (0,-2) -- (0, 5);

    \node at (\ldoc, -2.3) {Position ($\mathbb{R}$)};
    \node at (-0.8, 5) {Weight};

\end{tikzpicture}